\documentclass[12pt,a4paper]{amsart}
\usepackage[margin = 1in]{geometry}
\usepackage{amssymb}
\usepackage{amsfonts}
\usepackage{amsthm}
\usepackage{amsmath}
\usepackage{amscd}
\usepackage{esint}
\usepackage{verbatim}
\usepackage{graphicx}
\usepackage{caption}

\newcommand\norm[1]{\left\lVert#1\right\rVert}
\newtheorem{thm}{Theorem}[section]
\newtheorem{lemma}[thm]{Lemma}
\newtheorem{prop}[thm]{Proposition}

\theoremstyle{definition}
\newtheorem{definition}[thm]{Definition}

\theoremstyle{remark}

\numberwithin{equation}{section}

\begin{document}
	
	\title{Global regularity and fast small scale formation for Euler patch equation in a disk}
	
	\author{Chao Li\\
		Department of Mathematics\\
		Rice University}

	\begin{abstract}
		It is well known that the Euler vortex patch in $\mathbb{R}^{2}$ will remain regular if it is regular enough initially. In bounded domains, the regularity theory for patch solutions is less complete. We study here the Euler vortex patch in a disk. We prove global in time regularity by providing the upper bound of the growth of curvature of the patch boundary. For a special symmetric scenario, we construct an example of double exponential curvature growth, showing that such upper bound is qualitatively sharp.
	\end{abstract}

	\maketitle
	\pagestyle{plain}
	
	\section{Introduction}

    As we know, the Euler equation in the vorticity formulation in domain $D$ is given by 
	
	\begin{equation}\label{euler}
	\omega_{t} + (u\cdot\nabla)\omega = 0
	\end{equation}
	
	\begin{equation}\label{BSL}
	u = \nabla^{\perp}(-\Delta_{D})^{-1}\omega
	\end{equation}
	
	where $\nabla^{\perp}$ and $x^{\perp}$ denote $(\partial_{2},-\partial_{1})$ and $(x_{2},-x_{1})$ respectively, and $(-\Delta_{D})^{-1}\omega$ solves Laplace's equation with Dirichlet boundary condition $-\Delta\psi = \omega$ in $D$, with $\psi = 0$ on $\partial D$. Equation (\ref{BSL}) is called Biot-Savart law. \\
	
	Now we define particle trajectory associated to $u$ by 
	
    \begin{equation}\label{particletrajectory}
	\dfrac{d}{dt}\Phi_{t}(x) = u(\Phi_{t}(x),t), \quad \Phi_{0}(x) = x.
	\end{equation}
	
	Since the active scalar $\omega$ is transported by the velocity $u$ from (\ref{euler}), we have
	
	\begin{equation}\label{advection}
	\omega(x,t) = \omega_{0}(\Phi^{-1}_{t}(x)).
	\end{equation}
	
	The solution satisfying (\ref{BSL}), (\ref{particletrajectory}) and (\ref{advection}) is called a solution to Euler equation in Yudovich sense (see \cite{Yudovich}).
	
	An Euler vortex patch is a solution to the Euler equation in Yudovich sense of the form
	
	\begin{equation}
	  \omega(x,t) = \sum_{k = 1}^{N}\theta_{k}\chi_{\Omega_{k}(t)}(x).
	\end{equation}
	
	Here $\theta_{k}$ are some constants, and $\Omega_{k}(t)$ are (evolving in time) bounded open sets in $D$ with smooth (in some sense) boundaries, whose closure $\overline{\Omega_{k}(t)}$ are mutually disjoint.\\
	
	It is well known that in Yudovich sense, the solution to 
	2D Euler equation in $L^{\infty}\cap L^{1}$ exists and is unique (see \cite{Yudovich} or \cite{Majda} for a modern proof). In this paper, we study a stronger notion of regularity which refers to sufficient smoothness of the patch boundaries ($\partial\Omega_{k}$), as well as to the lack of both self-intersections of each patch boundary and touching of different patches.

	To be precise, we have the following series of definitions.
	
	\begin{definition}
		Let $\Omega \subseteq D$ be an open set whose boundary $\partial\Omega$ is a simple closed $C^{1}$ curve with arc-length $|\partial\Omega|$. A constant speed parametrization of $\partial\Omega$ is any counter-clockwise parametrization $z : \mathbb{T} \rightarrow \mathbb{R}^{2}$ of $\partial\Omega$ with $|z^{\prime}| = \frac{1}{2\pi}|\partial\Omega|$ on the circle $\mathbb{T} := [-\pi,\pi]$ (with $\pm\pi$ identified), and we define $\norm{\Omega}_{C^{m,\gamma}} :=  \norm{z}_{C^{m,\gamma}}$.
	\end{definition}
	
	\begin{definition}\label{generalpatchdef}
		Let $\theta_{1},...,\theta_{N} \in \mathbb{R}\setminus\{0\}$, and for each $t \in [0,T)$, let $\Omega_{1}(t),...,\Omega_{N}(t) \subseteq D$ be open sets with pairwise disjoint closures whose boundaries $\partial\Omega_{k}(t)$ are simple closed curves. Let
		\[
		\omega(x,t) := \sum\limits_{k=1}^{N}\theta_{k}\chi_{\Omega_{k}(t)}(x).
		\]
		Suppose $\omega$ also satisfies
		\[
		\omega(x,t) = \omega_{0}(\Phi^{-1}_{t}(x))
		\]
		where $\dfrac{d}{dt}\Phi_{t}(x) = u(\Phi_{t}(x),t)$, $\Phi_{0}(x) = x$, and $u$ is given by \ref{BSL}. Then $\omega$ is called a patch solution to (1.2)-(1.3) with initial data $\omega_{0}$ on the interval $[0,T)$. If in addition, we also have
		\[
		\sup\limits_{t \in [0,T^{\prime}]}\norm{\Omega_{k}(t)}_{C^{m,\gamma}} \, < \infty
		\]
		for each $k$ and $T^{\prime} \in (0,T)$, then $\omega$ is a $C^{m,\gamma}$ patch solution to (1.2) and (1.3) on $[0,T)$.
		
	\end{definition}
	
	\textbf{Remark} . In the above definition, the domains $\Omega_{k}(t)$ are allowed to touch $\partial D$ as long as $\partial\Omega_{k}(t)$ remains in $C^{m,\gamma}$.\\

	Singularity formation for two dimensional Euler vortex patches had been conjectured based on the numerical simulations by \cite{Buttke} for the first time (see \cite{AMajda} for a discussion). In 1993, Chemin \cite{Chemin} first proved that the boundary of a two dimensional Euler patch will remain regular for all time if it is regular enough ($C^{1,\gamma}$) initially, and the growth of its curvature is at most double exponential in time (see also the work by Bertozzi and Constantin in \cite{Constantin} for a different proof). For vortex patch in domains with boundaries, Depauw \cite{Depauw} has proved the global existence of a single $C^{1,\gamma}$ patch in the half plane when the patch does not touch the boundary initially. If the initial patch touches the boundary, then \cite{Depauw} proved that the regularity $C^{1,\gamma}$ will remain for a finite time. Sacrificing a little regularity, Dutrifoy \cite{Dutrifoy} proved that the global existence can be obtained in a weaker space $C^{1,s}$ for some $s \in (0,\gamma)$. Recently, Kiselev, Ryzhik, Yao and Zlato$\check{s}$ \cite{KRYZ} have proved the global regularity for two dimensional $C^{1,\gamma}$ Euler vortex patches (multiple) in half plane without loss of regularity.\\

	Our goal here is to explore Euler patch dynamics in bounded domains. We derive global upper bounds on the growth of curvature as well as construct examples showing sharpness of the upper bound in some instances. In this paper, we focus on the example of a disk, but we expect many techniques to extend the results to more general cases of domain with regular boundary. To avoid excessive technicalities, we leave this extension to future work.\\
	
    Throughout the paper (except Section 5), we denote the domain $D$ to be $B_{1}(0) = \{x \in \mathbb{R}^{2}: |x| < 1\}$.

	Here are our main results,

	\begin{thm}\label{singlepatch}
		Let $\gamma \in (0,1)$, then for each $C^{1,\gamma}$ single patch initial data $\omega_{0}$, there exists a unique global regular $C^{1,\gamma}$ patch solution $\omega$ to (1.1)-(1.2) with $\omega(\cdot,0) = \omega_{0}$. Furthermore, the curvature of boundary grows at most double exponentially.
	\end{thm}
	
	\begin{thm}\label{generalpatch}
		Let $\gamma \in (0,1)$, then for each $C^{1,\gamma}$ patch initial data $\omega_{0}$, there exists a unique global regular $C^{1,\gamma}$ patch solution $\omega$ to (1.1)-(1.2) with $\omega(\cdot,0) = \omega_{0}$. The curvature of boundary grows at most triple exponentially.	
	\end{thm}
	\textbf{Remark}. It is not clear if the triple exponential estimate is sharp. We have no concrete scenario for it, but at the same time improving this estimate requires non-trivial new ideas.\\

	In a special case where the initial patch consists of two symmetric single patches, we have a sharp upper bound estimate.
	
	\begin{thm}\label{symmetriccase}
		Let $\gamma \in (0,1)$, $\omega_{0}(x) = \chi_{\Omega_{1}}(x) - \chi_{\Omega_{2}}(x)$, where $\Omega_{1}$ and $\Omega_{2}$ are two single patches that are symmetric with respect to the line $x_{1}=0$. Then for each $C^{1,\gamma}$ initial data $\omega_{0}$, there exists a unique global regular $C^{1,\gamma}$ patch solution $\omega$ to (1.1)-(1.2) with $\omega(\cdot,0) = \omega_{0}$. The curvature of boundary grows at most double exponentially.	
	\end{thm}
	
	\begin{thm}\label{sharpexample}
		With the same assumptions as Theorem \ref{symmetriccase}, there exists an $\omega_{0}$ in $C^{1,\gamma}$ such that the curvature of the corresponding patch solution grows at a double exponential speed.
	\end{thm}
	
	The rest of the paper is organized as follows. In section 2, we give the proof of Theorem \ref{singlepatch}. In section 3, we deal with the general case, and provide the proof of Theorem \ref{generalpatch}. In section 4, we are looking into a special symmetric case, and prove Theorem \ref{symmetriccase}. In last section, we extend the example of \cite{KS} to show that the upper bound obtained in section 4 is actually sharp, thus proving Theorem \ref{sharpexample}.

	\section{Global regularity for single patch}
	We consider a single patch $\Omega(t)\subset D$, with
	\[
	\omega(x,t) = \theta _{0} \chi _{\Omega(t)}(x).
	\] 
	Without loss of generality, we set $\theta_{0} = 1$ throughout the section.\\
	
	Now, following \cite{Constantin}, we reformulate the vortex patch evolution in terms of the evolution of a function $\varphi(x,t)$, which defines the patch via
	\[
	\Omega(t) = \{x : \varphi(x,t) > 0\}.
	\]

	First, if $\partial\Omega(0)$ is a simple closed $C^{1,\gamma}$ curve, then there exists a function $\varphi_{0}\in C^{1,\gamma}(\overline{\Omega(0)})$, such that $\varphi_{0} > 0$ on $\Omega(0)$, $\varphi_{0} = 0$ on $\partial\Omega(0)$ and $\inf\limits_{\partial\Omega(0)}|\nabla\varphi_{0}|>0$. Such $\varphi_{0}$ can be obtained, for instance, by solving the Dirichlet problem
	
	\begin{tabbing}
		$\qquad$ $\qquad$ $\qquad$ $\qquad$ $-\Delta$ \= $\varphi_{0}$ $=$ $f$ on $\Omega(0)$,\\
		\> $\varphi_{0} = 0$ on $\partial\Omega(0)$,
	\end{tabbing}
	
	with arbitrary $0 \leq f \in C_{0}^{\infty}(\Omega(0))$(see the treatment in \cite{KRYZ}, or \cite{hardt} for a complete proof).\\

	Next, according to the Biot-Savart law and Green's function for $D$ (see, for example \cite{Trudinger}), we have

	\begin{equation}\label{BSL1}
	\begin{split}
	u(x,t) &= -\frac{1}{2\pi}\int_{\Omega(t)}\frac{(x-  y)^{\perp}}{|x-y|^{2}}dy + \frac{1}{2\pi}\int_{\widetilde{\Omega}(t)}\frac{(x-y)^{\perp}}{|x-y|^{2}}\frac{1}{|y|^{4}}dy\\
	&:= v(x,t) + \widetilde{v}(x,t),
	\end{split}
	\end{equation}
	where $\widetilde{\Omega}(t)$ is the image of $\Omega(t)$ reflected over $\partial D$, defined to be $\{\frac{x}{|x|^{2}}:x\in \Omega(t)\}$\\
	
	Before we start any calculation, we recall the flow map $\Phi_{t}(x)$ from (\ref{particletrajectory}) which describes the position of the particle at time $t$ starting from $x$. For $x\in \Omega(t)$, we set $\varphi(x,t) = \varphi_{0}(\Phi_{t}^{-1}(x))$, with $\Phi_{t}^{-1}$ being the inverse map of $\Phi_{t}$, so that $\varphi$ solves
	\[
	\partial_{t}\varphi + (u\cdot\nabla)\varphi = 0
	\]
	on $\{ (t,x) : t > 0 \ \textrm{and}\  x\in\Omega(t) \}$. Thus for each $t \geq 0$, $\varphi(\cdot,t) > 0$ on $\Omega(t)$, it vanishes on $\partial\Omega(t)$, and undefined on $\mathbb{R}^{2}\setminus\overline{\Omega(t)}$. Now we let
	\[
	w = (w_{1},w_{2})=\nabla^{\perp}\varphi=(\partial_{2}\varphi,-\partial_{1}\varphi),
	\]
	and define
	
	\begin{equation}\label{AAA}
	\begin{split} 
	A_{\gamma}(t) &:= \norm{w(\cdot,t)}_{\dot{C}^{\gamma}(\Omega(t))} = \sup\limits_{x,y\in\Omega(t)}\frac{|w(x,t)-w(y,t)|}{|x-y|^{\gamma}}, \\ 
	A_{\infty}(t) &:= \norm{w(\cdot,t)}_{L^{\infty}(\Omega(t))},\\
	A_{\inf}(t) &:= \inf\limits_{x\in\partial\Omega(t)}|w(x,t)|.
	\end{split}
	\end{equation}
	
	By our choice of $\varphi_{0}$, we have
	\[
	A_{\gamma}(0), \: A_{\infty}(0),\: A_{\inf}^{-1}(0) \: < \infty.
	\]
	
	Since $w = \nabla^{\perp}\varphi$, we know $w$ is divergence free and solves
	\begin{equation}\label{w-evolution}
	w_{t} + (u\cdot\nabla)w = (\nabla u)w.
	\end{equation}
	
	First we need a claim similar to Proposition 1 in \cite{Constantin}:
	\[
	\norm{\nabla u(\cdot,t)}_{L^{\infty}(\mathbb{R}^{2})} \:\leq\: C_{\gamma}\Big(1+\log_{+}\frac{A_{\gamma}(t)}{A_{\inf}(t)} \Big)
	\]
	with $\log_{+}(x) = \max\{\log(x),0\}$ and a constant $C_{\gamma}$ only depending on $\gamma$ that changes from line to line throughout the paper.\\
	
	Note that by (\ref{BSL1}), $u(x,t) = v(x,t) + \widetilde{v}(x,t)$. By Proposition 1 in \cite{Constantin}, $\norm{\nabla v(\cdot,t)}_{L^{\infty}(\mathbb{R}^{2})}$ can be bounded by $C_{\gamma}\Big(1+\log_{+}\dfrac{A_{\gamma}(t)}{A_{\inf}(t)} \Big)$.\\
	
	For $\widetilde{v}(x,t)$, first note

    \begin{equation}\label{tildev}
	\begin{split}
		\nabla\widetilde{v}(x,t) & =  \begin{pmatrix}
		\partial_{1}\widetilde{v_{1}} & \partial_{2}\widetilde{v_{1}} \\
		\partial_{1}\widetilde{v_{2}} & \partial_{2}\widetilde{v_{2}} 
		\end{pmatrix}\\
		& =  \frac{1}{2\pi}p.v.\int_{\widetilde{\Omega}(t)}\frac{\sigma(x-y)}{|x-y|^{2}}\frac{1}{|y|^{4}}dy\\
		&  \quad + \frac{1}{2|x|^{4}}\chi_{\widetilde{\Omega}(t)}(x)\begin{pmatrix}
		0 & -1\\
		1 & 0
		\end{pmatrix},
	\end{split}
    \end{equation}
	
	where $\sigma(x) = \begin{pmatrix}
	\dfrac{-2x_{1}x_{2}}{|x|^{2}} & 
	\dfrac{x_{1}^{2}-x_{2}^{2}}{|x|^{2}}\\
	\dfrac{x_{1}^{2}-x_{2}^{2}}{|x|^{2}} &
	\dfrac{2x_{1}x_{2}}{|x|^{2}}
	\end{pmatrix}$.
	
	Observe that $\dfrac{1}{2|x|^{4}}\chi_{\widetilde{\Omega}(t)}(x)\begin{pmatrix}
	0 & -1\\
	1 & 0
	\end{pmatrix}$ can be bounded by a universal constant. since $\widetilde{\Omega}(t) $For the integral term, if $x \in B_{1/2}(0)$, $|x-y| \geq 1/2$, it can be easily bounded by a universal constant. If $x \notin B_{1/2}(0)$, we split the integral domain, denote
	\[
	I_{1}(x) = \dfrac{1}{2\pi}p.v.\displaystyle\int_{\widetilde{\Omega}(t)\cap\{|x-y|\geq\delta\}}\frac{\sigma(x-y)}{|x-y|^{2}}\frac{1}{|y|^{4}}dy,
	\]
	\[
	I_{2}(x) = \dfrac{1}{2\pi}p.v.\displaystyle\int_{\widetilde{\Omega}(t)\cap\{|x-y|< \delta\}}\frac{\sigma(x-y)}{|x-y|^{2}}\frac{1}{|y|^{4}}dy,
	\]

	 here we follow the argument in \cite{Constantin} and set $\delta^{\gamma} = \dfrac{A_{\inf}}{A^{\gamma}}$.\\
	
	\begin{tabbing}
		\hspace{1cm}$I_{1}(x)$ \= $\leq$ \=$\dfrac{1}{2\pi}\displaystyle\int_{\widetilde{\Omega}(t)\cap\{|x-y|\geq\delta\}}\frac{1}{|x-y|^{2}}\frac{1}{|y|^{4}}dy$\\
		\> $\leq$ \> $\dfrac{1}{2\pi}\displaystyle\int_{\widetilde{\Omega}(t)\cap\{\delta\leq |x-y|\leq 2\}}\frac{1}{|x-y|^{2}}\frac{1}{|y|^{4}}dy$\\
		\> \> $+$ $\dfrac{1}{2\pi}\displaystyle\int_{\{|y|\geq 1 \}\cap\{|x-y|>2\}}\frac{1}{|x-y|^{2}}\frac{1}{|y|^{4}}dy$\\
		\> $\leq$ \> $\dfrac{1}{2\pi}\Big(\log\big(\frac{2}{\delta}\big)+\frac{1}{8}\Big).$
	\end{tabbing}
	In the above estimate, we assume $\delta \leq 2$. If $\delta > 2$, $I_{1}(x)$ can be bounded by a universal constant, which is better.\\
	
	For $I_{2}(x)$, we need to consider two cases. First assume $\delta \geq 1$, then we have
	\begin{tabbing}
		\hspace{1cm}$I_{2}(x)$ \= $=$ \= $\dfrac{1}{2\pi}p.v.\displaystyle\int_{\widetilde{\Omega}(t)\cap\{|x-y|< \delta\}}\frac{\sigma(x-y)}{|x-y|^{2}}\frac{1}{|y|^{4}}dy$\\
		\> $=$ \> $\dfrac{1}{2\pi}p.v.\displaystyle\int_{\widetilde{\Omega}(t)\cap\{|x-y| < 1\}}\frac{\sigma(x-y)}{|x-y|^{2}}\Big(\frac{1}{|y|^{4}}-\frac{1}{|x|^{4}}\Big)dy$\\
		\> \> $+$ $\dfrac{1}{2\pi}p.v.\displaystyle\int_{\widetilde{\Omega}(t)\cap\{|x-y| < 1\}}\frac{\sigma(x-y)}{|x-y|^{2}}\frac{1}{|x|^{4}}dy$\\
		\> \> $+$ $\dfrac{1}{2\pi}\displaystyle\int_{\widetilde{\Omega}(t)\cap\{1 \leq|x-y|< \delta\}}\frac{\sigma(x-y)}{|x-y|^{2}}\frac{1}{|y|^{4}}dy$\\
		\> $=$ \> $I_{21}(x) + I_{22}(x) + I_{23}(x)$.
	\end{tabbing}
	For $I_{21}(x)$, the singularity can be killed by an extra $|x-y|$ rising from $\Big(\dfrac{1}{|y|^{4}}-\dfrac{1}{|x|^{4}}\Big)$ by mean value theorem. So $I_{21}(x)$ can be bounded by a universal constant. For $I_{23}(x)$, as $|x-y| \geq 1$, $|y| \geq 1$, it can be bounded by a universal constant, too. For $I_{22}(x)$, as $\dfrac{1}{|x|^{4}}$ is a constant in the integral, all we need to estimate is $p.v.\displaystyle\int_{\widetilde{\Omega}(t)\cap\{|x-y| < 1\}}\frac{\sigma(x-y)}{|x-y|^{2}}dy$. Here we claim that this is bounded by a constant $C_{\gamma}$ depending on $\gamma$. A complete proof can be found in \cite{Constantin}, note that this is under the assumption $\delta \geq 1$.\\ 
	
	On the other hand, if $\delta < 1$
	\begin{tabbing}
		\hspace{1cm}$I_{2}(x)$ \= $=$ \= $\dfrac{1}{2\pi}\displaystyle\int_{\widetilde{\Omega}(t)\cap\{|x-y|< \delta\}}\frac{\sigma(x-y)}{|x-y|^{2}}\frac{1}{|y|^{4}}dy$\\
		\> $=$ \> $\dfrac{1}{2\pi}\displaystyle\int_{\widetilde{\Omega}(t)\cap\{|x-y| < \delta\}}\frac{\sigma(x-y)}{|x-y|^{2}}\Big(\frac{1}{|y|^{4}}-\frac{1}{|x|^{4}}\Big)dy$\\
		\> \> $+$ $\dfrac{1}{2\pi}\displaystyle\int_{\widetilde{\Omega}(t)\cap\{|x-y| < \delta\}}\frac{\sigma(x-y)}{|x-y|^{2}}\frac{1}{|x|^{4}}dy$\\
		\> $=$ \> $I^{\prime}_{21}(x) + I^{\prime}_{22}(x).$
	\end{tabbing}
	First note that $I^{\prime}_{21}(x) \leq \dfrac{1}{2\pi}\displaystyle\int_{\widetilde{\Omega}(t)\cap\{|x-y| < 1\}}\frac{1}{|x-y|^{2}}\Big|\frac{1}{|y|^{4}}-\frac{1}{|x|^{4}}\Big|dy$, similar to the estimate of $I_{21}(x)$, it is bounded by a universal constant. $I^{\prime}_{22}(x)$ can be bounded by a constant $C_{\gamma}$ depending only on $\gamma$, a complete estimate can be found in \cite{Constantin}.\\
	
	Hence, we have proved the following bound.
	\begin{prop}\label{gradient o fvelocity}
		Assume $u$ is given by the Biot-Savart law formula (\ref{BSL1}) and $A_{\gamma}$, $A_{\infty}$, $A_{\inf}$ are defined  by (\ref{AAA}). Then,
		\begin{equation}
		\norm{\nabla u(\cdot,t)}_{L^{\infty}(\mathbb{R}^{2})} \:\leq\: C_{\gamma}\Big(1+\log_{+}\frac{A_{\gamma}(t)}{A_{\inf}(t)} \Big).
		\end{equation}
	\end{prop}
	
	Next, we obtain from equation (\ref{w-evolution}),
	
	\begin{align}\label{Ainftyestimate}
	A^{\prime}_{\infty}(t) & \leq C_{\gamma}A_{\infty}(t)\Big(1+\log_{+}\frac{A_{\gamma}(t)}{A_{\inf}(t)} \Big),\\
	\label{Ainfestimate}
	A^{\prime}_{\inf}(t) & \geq -C_{\gamma}A_{\inf}(t)\Big(1+\log_{+}\frac{A_{\gamma}(t)}{A_{\inf}(t)} \Big).
	\end{align}
	
	The main step in the proof will be to get an appropriate bound on $A_{\gamma}$. A simple calculation and (\ref{w-evolution}) yield,
	
	\begin{equation}
	A^{\prime}_{\gamma}(t) \leq C_{\gamma}\norm{\nabla u(\cdot,t)}_{L^{\infty}(\mathbb{R}^{2})}A_{\gamma}(t) + \norm{\nabla u(\cdot,t)w(\cdot,t)}_{\dot{C}^{\gamma}(\Omega(t))}.
	\end{equation}
	
	Later we will show that
	\begin{equation}\label{Agammaestimate}
	A_{\gamma}^{\prime}(t) \leq C_{\gamma}A_{\gamma}(t)\Big(1+\log_{+}\dfrac{A_{\gamma}(t)}{A_{\inf}(t)} \Big) + A_{\infty}(t)
	\end{equation}
	by claiming that 
	
	\begin{equation}
	\norm{\nabla u(\cdot,t)w(\cdot,t)}_{\dot{C}^{\gamma}(\Omega(t))} \leq C_{\gamma}A_{\gamma}(t)\Big(1+\log_{+}\dfrac{A_{\gamma}(t)}{A_{\inf}(t)} \Big) + A_{\infty}(t).
	\end{equation}
	
	Assuming (\ref{Ainftyestimate}), (\ref{Ainfestimate}) and (\ref{Agammaestimate}), we can prove Theorem \ref{singlepatch} now.
	
	\begin{proof}[Proof of Theorem \ref{singlepatch}]
	
	Let $A(t) := \dfrac{A_{\gamma}(t) + A_{\infty}(t)}{A_{\inf}(t)}$, we have, from above estimates,
	\[
	A^{\prime}(t) \leq C_{\gamma}A(t)(1 + \log_{+}A(t)).
	\]
	Thus obtaining that $A(t)$ grows at most double-exponentially in time, and therefore, so is $\dfrac{A_{\gamma}(t)}{A_{\inf}(t)}$. Then the same rate of growth can be obtained for $A_{\infty}(t)$, $A_{\inf}(t)^{-1}$ and $A_{\gamma}(t)$ from (\ref{Ainftyestimate}), (\ref{Ainfestimate}) and (\ref{Agammaestimate}) respectively. Thus completing the proof.
	\end{proof}

	Now our goal is to derive a desired bound for $\norm{(\nabla u)w}_{\dot{C}^{\gamma}(\Omega(t))}$.\\

	Note that
	\[
	\norm{(\nabla u)w}_{\dot{C}^{\gamma}(\Omega(t))} \leq \norm{(\nabla v)w}_{\dot{C}^{\gamma}(\Omega(t))} + \norm{(\nabla\widetilde{v})w}_{\dot{C}^{\gamma}(\Omega(t))}.
	\] 
	
	Since $v$ is generated by the patch $\Omega(t)$, and $w$ is tangent to $\partial\Omega(t)$, Corollary 1 in \cite{Constantin} gives
	\[
	\norm{(\nabla v)w}_{\dot{C}^{\gamma}(\Omega(t))} \leq  C_{\gamma}\norm{\nabla v}_{L^{\infty}(\mathbb{R}^{2})}\norm{w}_{\dot{C}^{\gamma}(\Omega(t))}
	\]
	with a universal constant $C_{\gamma}$ only depending on $\gamma$. Note that in \cite{Constantin}, $w$ is defined in $\mathbb{R}^{2}$ and all the norms are over $\mathbb{R}^{2}$. We can use Whitney-type extension theorem (see page 170 in \cite{Stein}) to extend our $\varphi$ to all of $\mathbb{R}^{2}$ so that its $C^{1,\gamma}$ norm increases at most by a universal factor $C_{\gamma}$ depending only on $\gamma$. From now on $\varphi$ and $w$ are understood to be defined on $\mathbb{R}^{2}$.\\

	What remains now is to derive a similar estimate on the second term. Since the estimate is time independent, for convenience, we will drop $t$ and use $\Omega$ to replace $\Omega(t)$ from now on.

	\begin{figure}[h]
		\includegraphics[scale=0.5]{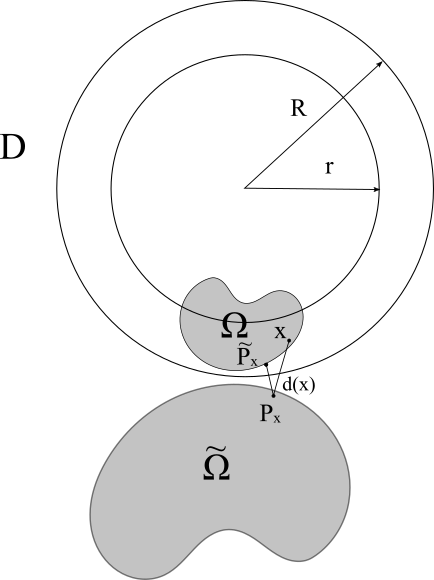}
		\caption{}
		\label{figure1}
	\end{figure}

	First we split $D$ as a smaller ball $B_{r}(0)$ and an annulus $A$, see Figure \ref{figure1}. When $x \in B_{r}(0)$ in the expression of $\nabla\widetilde v$ by (\ref{tildev}), $|x - y| \geq 1-r$, the velocity is smooth enough. When $x \in A$, the reflection over $D$ is very close to the reflection over a line, which helps avoid some complications rising from the Green's function.\\

	For later use, we choose the smaller ball to be $B_{10/11}(0))$, the annulus to be $A(0;9/10,1)$, where $A(0;9/10,1) := \{x: 9/10 \leq |x| \leq 1\}$. Here, we introduce a little overlap, because we want the estimate of $\norm{(\nabla\widetilde{v})w}_{\dot{C}^{\gamma}(\Omega \cap B_{10/11}(0))}$ and $\norm{(\nabla\widetilde{v})w}_{\dot{C}^{\gamma}(\Omega \cap A(0;9/10,1)}$ to be sufficient to bound $\norm{(\nabla\widetilde{v})w}_{\dot{C}^{\gamma}(\Omega)}$. \\

	For the easy part, we have 
	\begin{tabbing}
		$\norm{(\nabla\widetilde{v})w}_{\dot{C}^{\gamma}(\Omega \cap B_{10/11}(0))}$ \= $\leq$ $\norm{\nabla\widetilde{v}(\cdot,t)}_{L^{\infty}(B_{10/11}(0))}\norm{w(\cdot,t)}_{\dot{C}^{\gamma}(B_{10/11}(0))}$\\
		\> $+$ $\norm{w(\cdot,t)}_{L^{\infty}(B_{10/11}(0))}\norm{\nabla\widetilde{v}(\cdot,t)}_{\dot{C}^{\gamma}(B_{10/11}(0))}$\\
		\> $\leq$ $C_{\gamma}(A_{\gamma} + A_{\infty}).$
	\end{tabbing}    
	the above estimate can be obtained by absolute value bounds, since the velocity away from the boundary of $D$ is smooth.\\
	
	To deal with the second part, we first define $\widetilde{\varphi}(x)$ and $\widetilde{w}(x)$ as follows: 
	\begin{equation}\label{tildevarphi}
	\widetilde{\varphi}(x) := \varphi(\widetilde{x}), \quad \widetilde{x} = \frac{x}{|x|^{2}},
	\end{equation}
	\begin{equation}\label{definitionofw}
	\widetilde{w}(x) := \nabla^{\perp}\widetilde{\varphi}(x) = \begin{pmatrix}
	\dfrac{x_{1}^{2}-x_{2}^{2}}{|x|^{4}} &
	\dfrac{2x_{1}x_{2}}{|x|^{4}}\\
	\dfrac{2x_{1}x_{2}}{|x|^{4}} &
	\dfrac{x_{2}^{2}-x_{1}^{2}}{|x|^{4}} \\
	\end{pmatrix}
	\begin{pmatrix}
	w_{1}(\widetilde{x}) \\
	w_{2}(\widetilde{x})
	\end{pmatrix}.
	\end{equation}
	
	Here we first list the proposition that we will prove below.
	
	\begin{prop}\label{anulusestimate}
		Let $\varphi$, $\widetilde{v}$, $\widetilde{w}$, $w$, $A_{\gamma}$, $A_{\inf}$ be as above. Then we have
		\[
		\norm{(\nabla\widetilde{v})w}_{\dot{C}^{\gamma}(\Omega \cap A(0;9/10,1))} \leq C_{\gamma}A_{\gamma}\Big(1+\log_{+}\frac{A_{\gamma}(t)}{A_{\inf}(t)} \Big).
		\]
	\end{prop}
	
	Let us introduce some notations first. See Figure \ref{figure1}, for any $x \in \mathbb{R}^{2}\setminus\widetilde{\Omega}$, define $d(x) := \text{dist}(x,\widetilde{\Omega})$. Let $P_{x} \in \partial\widetilde{\Omega}$ be the point such that $d(x) = \text{dist}(x,P_{x})$ (if there are multiple such points, we pick any one of them), and let $\widetilde{P}_{x}$ be the reflection point of $P_{x}$ over the boundary of $D$.\\

	For any two points $x, y \in \Omega\cap A(0;9/10,1)$, first note that $d(x), d(y) \leq 1/4$. Then we assume, without loss of generality, that $d(x) \leq d(y)$. With $g := (\nabla\widetilde{v})w$, we have
	
	\begin{equation}\label{govergamma}
	\begin{split}
		 \dfrac{|g(x)-g(y)|}{|x-y|^{\gamma}} & \leq |\nabla\widetilde{v}(y)|\norm{w}_{\dot{C}^{\gamma}(\Omega\cap A(0;9/10,1))}\\
		& \quad + \dfrac{|\nabla\widetilde{v}(x) - \nabla\widetilde{v}(y)}{|x-y|^{\gamma}}|w(x)|.
	\end{split}
	\end{equation}

	The first term on the right hand side is bounded by $C_{\gamma}A_{\gamma}\Big(1+\log_{+}\dfrac{A_{\gamma}(t)}{A_{\inf}(t)} \Big)$ due to Proposition \ref{gradient o fvelocity}. So we only need to bound the second term.\\
	
	Note that
	\begin{equation}\label{omegapx}
	\begin{split}
		|w(x)| & \leq |w(\widetilde{P}_{x})|+|w(\widetilde{P}_{x})-w(x)|\\
		& \leq |w(\widetilde{P}_{x})| + C_{\gamma}A_{\gamma}d(x)^{\gamma}.
	\end{split}
	\end{equation}
	
	This is because $d(x) \geq \text{dist}(P_{x}, D)$, and $x \in \Omega\cap A(0;9/10,1)$ implies that $\text{dist}(P_{x},D)$ and $\text{dist}(\widetilde{P}_{x},D)$ are comparable. Therefore we can find a universal constant $C > 0$, such that $|P_{x} - \widetilde{P}_{x}| \leq Cd(x)$, $|x - \widetilde{P}_{x}| \leq Cd(x)$.\\
	
	Next we are going to estimate $\dfrac{|\nabla\widetilde{v}(x) - \nabla\widetilde{v}(y)|}{|x-y|^{\gamma}}$ in (\ref{govergamma}). The following are Propositions and Lemmas that we need.
	
	\begin{prop}\label{T2estimate}
		Under the assumption of Proposition \ref{anulusestimate}, for $x,y\in\Omega\cap A(0;9/10,1)$, with $d(x) \leq d(y)$. Then,
		\[
		\frac{|\nabla\widetilde{v}(x) - \nabla\widetilde{v}(y)|}{|x-y|^{\gamma}} \leq C_{\gamma}\Big(1+\log_{+}\frac{A_{\gamma}}{A_{\inf}}\Big)\min\Big\{\frac{A_{\gamma}}{|\widetilde{w}(P_{x})|}, d(x)^{-\gamma}\Big\}.
		\]
	\end{prop}
	
	\begin{lemma}\label{dgamma}
		For $x,y\in\Omega\cap A(0;9/10,1)$, with $d(x) \leq d(y)$. We have (with a universal constant $C < \infty$)
		\[
		\frac{|\nabla\widetilde{v}(x) - \nabla\widetilde{v}(y)|}{|x-y|^{\gamma}} \leq C_{\gamma}d(x)^{-\gamma}.
		\]
	\end{lemma}

	\begin{lemma}\label{logestimate}
		For $x,y\in\Omega\cap A(0;9/10,1)$, with $d(x) \leq \min\{d(y),2^{-4-1/\gamma}r_{x}\}$, where $r_{x} := \Big(\dfrac{|\widetilde{w}(P_{x})|}{2A_{\gamma}}\Big)^{\frac{1}{\gamma}}$. We have (with a constant $C_{\gamma}$ only depending on $\gamma$)
		\[
		\frac{|\nabla\widetilde{v}(x) - \nabla\widetilde{v}(y)|}{|x-y|^{\gamma}} \leq C_{\gamma}\Big(1+\log_{+}\frac{A_{\gamma}}{A_{\inf}}\Big)\frac{A_{\gamma}}{|\widetilde{w}(P_{x})|}.
		\]
	\end{lemma}
	
	First, we prove Proposition \ref{anulusestimate} by using Proposition \ref{T2estimate}, Lemma \ref{dgamma} and Lemma \ref{logestimate}.

	\begin{proof}[Proof of Proposition \ref{anulusestimate}]
		By (\ref{omegapx}), we know $|w(x)| \leq |w(\widetilde{P}_{x})| + C_{\gamma}A_{\gamma}d(x)^{\gamma}$. By definition (\ref{definitionofw}), we know that for $x \in \Omega\cap A(0;9/10,1)$, $|w(\widetilde{P}_{x})| \leq C|\widetilde{w}(P_{x})|$, where $C$ is a universal constant. Together with Proposition \ref{T2estimate}, we have,
		\begin{tabbing}
			$\dfrac{|\nabla\widetilde{v}(x) - \nabla\widetilde{v}(y)}{|x-y|^{\gamma}}|w(x)|$ \= $\leq$  $C_{\gamma}A_{\gamma}d(x)^{\gamma}\Big(1+\log_{+}\dfrac{A_{\gamma}}{A_{\inf}}\Big)\min\Big\{\dfrac{A_{\gamma}}{|\widetilde{w}(P_{x})|}, d(x)^{-\gamma}\Big\}$\\
			\> $\quad$ $+$ $|\widetilde{w}(P_{x})|\Big(1+\log_{+}\dfrac{A_{\gamma}}{A_{\inf}}\Big)\min\Big\{\dfrac{A_{\gamma}}{|\widetilde{w}(P_{x})|}, d(x)^{-\gamma}\Big\}$\\
			\> $\leq$ $C_{\gamma}A_{\gamma}\Big(1+\log_{+}\dfrac{A_{\gamma}}{A_{\inf}}\Big).$
		\end{tabbing}
		
	\end{proof}

	Next, we prove the proposition \ref{T2estimate} by using Lemma \ref{dgamma} and Lemma \ref{logestimate}.

	\begin{proof}[Proof of Proposition \ref{T2estimate}]
		Due to Lemma \ref{dgamma}, we only need to consider the case where
		\[
		d(x) \leq \Big(\frac{|\widetilde{w}(P_{x})|}{A_{\gamma}}\Big)^{\frac{1}{\gamma}}.
		\]
		Lemma \ref{logestimate} completes the proof.
	\end{proof}

	Next, we prove Lemma \ref{dgamma} and Lemma \ref{logestimate}.

	\begin{proof}[Proof of Lemma \ref{dgamma}]
		By mean value theorem,
		\[
		\frac{|\nabla\widetilde{v}(x) - \nabla\widetilde{v}(y)|}{|x-y|^{\gamma}} \leq |\nabla^{2}\widetilde{v}(z)||x-y|^{1-\gamma}
		\]
		for some point $z$ on the segment connecting $x$ and $y$. Let $d_{x}$ be the distance between $x$ and $\partial D$. By definition, we always have $d(x) \geq d_{x}$. Since $x,y\in\Omega\cap A(0;9/10,1)$, we have (with a universal constant $C < \infty$) 
		
		\begin{equation}\label{relation1}
		\begin{split}
		d_{x} & \leq d(x) \leq Cd_{x},\\
		d_{y} & \leq d(y) \leq Cd_{y},\\
		 &  d_{z} \leq d(z).
		\end{split}
		\end{equation}
		So
		\begin{equation}\label{relation3}
		d(z) \geq d_{z} \geq \min\{d_{x},d_{y}\} \geq C\min\{d(x),d(y)\} \geq Cd(x).
		\end{equation}
		Moreover, for any point $Z \not\in\widetilde{\Omega}$, we have (with a universal constant $C < \infty$)
		\[
		|\nabla^{2}\widetilde{v}(Z)| \leq \int_{\mathbb{R}^{2}\setminus B_{d(Z)}(Z)}\frac{C}{|Z-x|^{3}}dx \leq Cd(Z)^{-1}.
		\]
		Putting together, we obtain
		\[
		\frac{|\nabla\widetilde{v}(x) - \nabla\widetilde{v}(y)|}{|x-y|^{\gamma}} \leq Cd(z)^{-1}|x-y|^{1-\gamma} \leq Cd(x)^{-1}|x-y|^{1-\gamma}.
		\]
		So if $|x-y| \leq d(x)$, the proof is done.\\
		
		On the other hand, we assume $|x-y| \geq d(x)$. Since $x,y\in\Omega\cap A(0;9/10,1)$, we have $d(x), d(y) \leq 1/4$. We can further assume that $|x-y| \leq 1/4$. Indeed, with $x,y \in D$, $|x -y| \leq 2$. We can always insert a finite number of points to make each two consecutive points with a distance less than $1/4$.\\

		\begin{figure}[h]
		\includegraphics[scale=0.5]{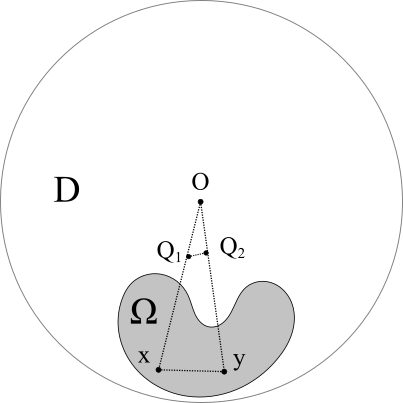}
		\caption{}
		\label{figure2}
	    \end{figure}

		Now let $Q_{1}$ be the point on the segment $Ox$ connecting point $x$ and the origin $O$ of $D$, with a distance of $2|x-y|$ from $x$, let $Q_{2}$ be the point on segment $Oy$ with a distance of $2|x-y|$ from $y$, see Figure \ref{figure2}. For $s \in [0,1]$, we have

		\begin{equation}\label{relation2}
		\begin{split}
		d(x+s(Q_{1}-x)) & \geq \max\{d(x),2s|x-y|\},\\
		d(y+s(Q_{2}-y)) & \geq \max\{d(x),2s|x-y|\},\\
		d(Q_{1} + s(Q_{2} - Q_{1})) & \geq |x - y|,\\
		|Q_{1}-Q_{2}| & \leq 3|x-y|.
		\end{split}
		\end{equation}

		Now integrating along the path $x \rightarrow Q_{1} \rightarrow Q_{2} \rightarrow y$, we have (with a universal constant $C < \infty$)
		\begin{tabbing}
			$\quad |\nabla\widetilde{v}(x) - \nabla\widetilde{v}(y)|$ \= $\leq$ \= $2\displaystyle\int_{0}^{1}|\nabla^{2}\widetilde{v}(x+s(Q_{1}-x))||x-y|ds$\\
			\> \>$+$ $\displaystyle\int_{0}^{1}|\nabla^{2}\widetilde{v}(Q_{1}+s(Q_{2}-Q_{1}))||Q_{1}-Q_{2}|ds$\\
			\> \>$+$ $2\displaystyle\int_{0}^{1}|\nabla^{2}\widetilde{v}(y+s(Q_{2}-y))||x-y|ds$\\
			\> $\leq$ \> $C|x-y|\big(\displaystyle\int_{0}^{\frac{d(x)}{|x-y|}}d(x)^{-1}ds + \displaystyle\int_{\frac{d(x)}{|x-y|}}^{1}(s|x-y|)^{-1}ds\big)$\\
			\> \> $+$ $C|x-y|^{-1}|x-y|$\\
			\>$\leq$ \> $C(1+\log\dfrac{|x-y|}{d(x)}).$\\
		\end{tabbing}
		Now we have (with a constant $C_{\gamma}$ only depending on $\gamma$)
		\[
		\frac{|\nabla\widetilde{v}(x) - \nabla\widetilde{v}(y)|}{|x-y|^{\gamma}} \leq C(1+\log\frac{|x-y|}{d})|x-y|^{-\gamma} \leq C_{\gamma}d^{-\gamma}.
		\]
		Indeed, for $a \geq 1$, we always have $1 + \log a \leq \frac{1}{\gamma}a^{\gamma}$.
	\end{proof}
	
	To prove Lemma \ref{logestimate}, we need Lemma \ref{secondgradient} stated as follows.

	\begin{lemma}\label{secondgradient}
		For any $x \in \mathbb{R}^{2}\setminus\widetilde{\Omega}$ with $d(x) \leq \frac{1}{4}r_{x}$, and $r_{x} := \Big(\dfrac{|\widetilde{w}(P_{x})|}{2A_{\gamma}}\Big)^{\frac{1}{\gamma}}$. We have 
		\[
		|\nabla^{2}\widetilde{v}(x)| \leq C_{\gamma}d(x)^{-1+\gamma}r_{x}^{-\gamma}.
		\]
	\end{lemma}
	
 	Let us prove Lemma \ref{logestimate} first by assuming Lemma \ref{secondgradient} is true.
	
	\begin{proof}[Proof of Lemma \ref{logestimate}]
		If $|x-y| \geq 2^{-4-1/\gamma}r_{x}$, then we have $|x-y|^{-\gamma} \leq C_{\gamma}\dfrac{A_{\gamma}}{|\widetilde{w}(P_{x})|}$. The proof follows directly from
		\[
		|\nabla\widetilde{v}(x) - \nabla\widetilde{v}(y)| \leq 2\norm{\nabla\widetilde{v}}_{L^{\infty}(\mathbb{R}^{2})}.
		\]
		
		On the other hand, assume that $|x-y| < 2^{-4-1/\gamma}r_{x}$. Without loss of generality, we can further assume $|x-y| \leq 1/4$. Indeed, for $x, y \in D$, we have $|x-y| \leq 2$, we can always insert a finite number of points to make each two consecutive points with a distance less than $1/4$, see Figure \ref{figure2}. Since $x,y\in A(0;9/10,1)$, we have $d(x),d(y) \leq 1/4$. Now we can choose $Q_{1}$ and $Q_{2}$ in the same way as in the proof of Lemma \ref{dgamma}. We parametrize the segments $xQ_{1}$, $Q_{1}Q_{2}$ and $Q_{2}y$ by
		\[
		z_{1}(s) = x + s(Q_{1}-x),
		\]
		\[
		z_{2}(s) = Q_{1} + s(Q_{2} - Q_{1}),
		\]
		\[
		z_{3}(s) = y + s(Q_{2} - y),
		\]
		where $s\in[0,1]$.\\
		
		For $i =1, 2, 3$ and $s\in[0,1]$, we also have
		\[
		|z_{i}(s)-P_{x}| \leq |z_{i}(s) -x| + d(x) \leq 3|x-y|+d(x).
		\]
		So 
		\begin{equation}\label{dvsrx}
		d(z_{i}(s)) \leq 2^{-2-1/\gamma}r_{x}.
		\end{equation}
		
		These inequalities imply
		\[
		P_{z_{i}(s)} \in B_{x} := B_{r_{x}}(P_{x}).
		\]
		Note also that we have 
		\[
		|\widetilde{w}(P_{z_{i}(s)}) - \widetilde{w}(P_{x})| \leq A_{\gamma}|P_{z_{i}(s)}-P_{x}|^{\gamma} \leq \frac{|\widetilde{w}(P_{x})|}{2},
		\]
		which gives us 
		\[
		|\widetilde{w}(P_{z_{i}(s)})| \geq \frac{|\widetilde{w}(P_{x})|}{2},
		\]
		implying 
		\[
		r_{z_{i}(s)} \geq 2^{-\frac{1}{\gamma}}r_{x}.
		\]
		From (\ref{dvsrx}) it follows that 
		\[
		d(z_{i}(s)) \leq \frac{1}{4}r_{z_{i}(s)}.
		\]
		Now we can apply Lemma \ref{secondgradient} to $z_{i}(s)$ and obtain
		\[
		|\nabla^{2}\widetilde{v}(z_{i}(s))| \leq C_{\gamma}d(z_{i}(s))^{-1+\gamma}r_{z_{i}(s)}^{-\gamma} \leq C_{\gamma}(s|x-y|)^{-1+\gamma}r_{x}^{-\gamma}
		\]
		Integrating along the path $x \rightarrow Q_{1} \rightarrow Q_{2} \rightarrow y$, we have
		\begin{tabbing}
			$\quad \dfrac{|\nabla\widetilde{v}(x) - \nabla\widetilde{v}(y)|}{|x-y|^{\gamma}}$ \= $\leq$ \= $2\displaystyle\int_{0}^{1}|\nabla^{2}\widetilde{v}(x+s(Q_{1}-x))||x-y|^{1-\gamma} ds$\\
			\> \>$+$ $\displaystyle\int_{0}^{1}|\nabla^{2}\widetilde{v}(Q_{1}+s(Q_{2}-Q_{1}))||x-y|^{1-\gamma}ds$\\
			\> \>$+$ $2\displaystyle\int_{0}^{1}|\nabla^{2}\widetilde{v}(y+s(Q_{2}-y))||x-y|^{1-\gamma}ds$\\
			\> $\leq$ \> $C|x-y|^{1-\gamma}\displaystyle\int_{0}^{1}(s|x-y|)^{-1+\gamma}r_{x}^{-\gamma}ds$\\
			\> $\leq$ \> $C_{\gamma}r_{x}^{-\gamma}$ $\leq$ $C_{\gamma}\dfrac{A_{\gamma}}{|\widetilde{w}(P_{x})|}.$
		\end{tabbing}
		
	\end{proof}

	\begin{figure}[h]
		\includegraphics[scale=0.5]{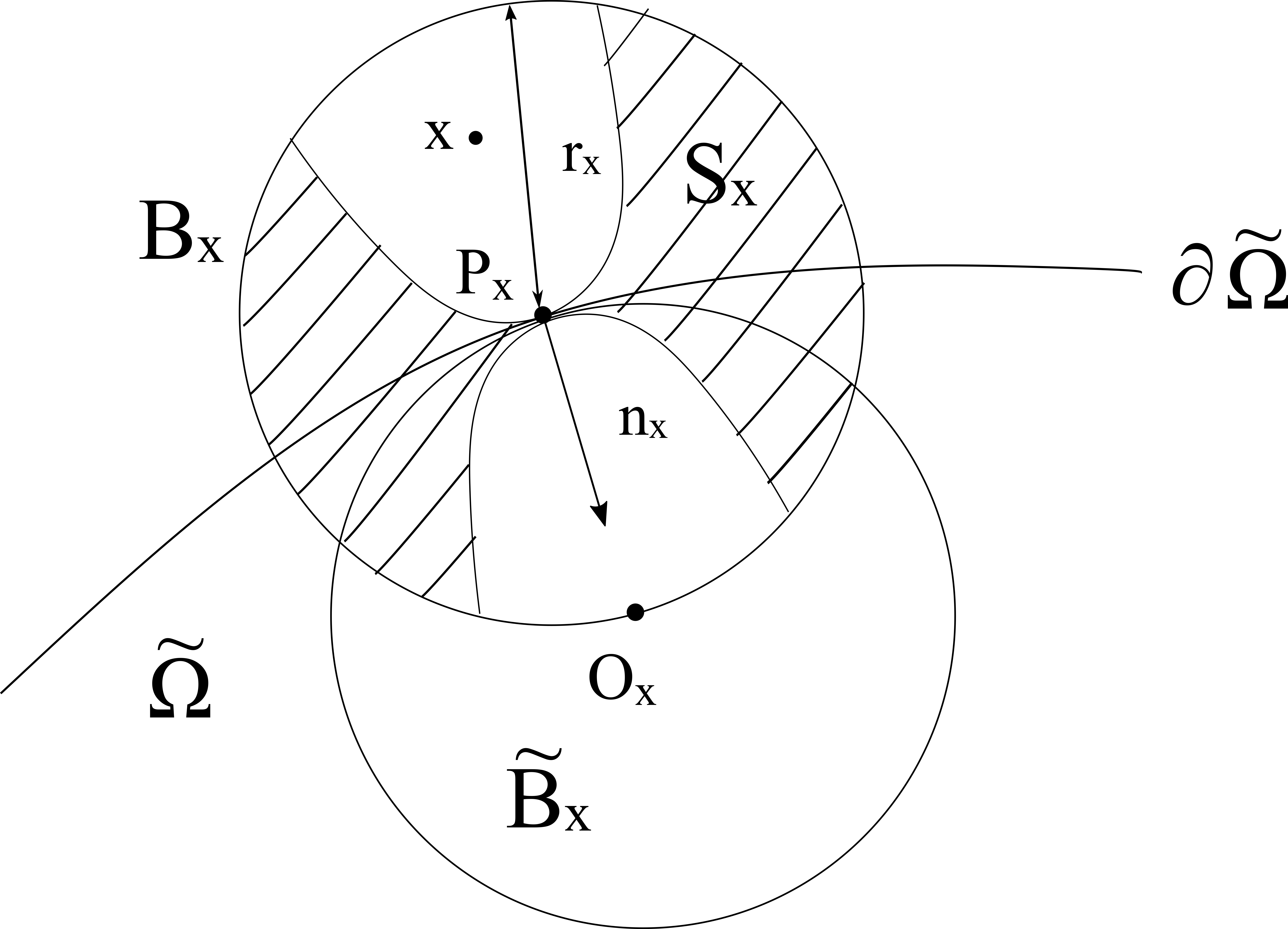}
		\caption{}
		\label{figure3}
	\end{figure}

	Now our goal is to prove Lemma \ref{secondgradient}. To achieve this, we need a result from \cite{KRYZ}, which basically says that $\partial\widetilde{\Omega}$ is sufficiently "flat" near $P_{x}$, for $x \notin \widetilde{\Omega}$, see Figure \ref{figure3}.
	
	\begin{lemma}\label{geometriclemma}
		Given $x \in \mathbb{R}^{2}\setminus \widetilde{\Omega}$, let $n_{x} := \nabla\widetilde{\varphi}(P_{x})/|\nabla\widetilde{\varphi}(P_{x})|$, $r_{x} := \Big(\dfrac{|\widetilde{w}(P_{x})|}{2A_{\gamma}}\Big)^{\frac{1}{\gamma}}$, and 
		$S_{x} := \{P_{x}+\rho\nu : \rho \in [0,r_{x}), |\nu| = 1, (\frac{\rho}{r_{x}})^{\gamma} \geq 2|\nu\cdot n_{x}|\}$.
		If $\nu$ is a unit vector and $\rho \in [0,r_{x})$, then the following hold. If $\nu\cdot n_{x} \geq 0$ and $P_{x} + \rho\nu \not\in S_{x}$, then $P_{x} + \rho\nu \in \widetilde{\Omega}$. If $\nu\cdot n_{x} \leq 0$ and $P_{x} + \rho\nu \not\in S_{x}$, then $P_{x} + \rho\nu \in \mathbb{R}^{2}\setminus\widetilde{\Omega}$.
	\end{lemma}
	\textbf{Remark}. The proof of Lemma \ref{geometriclemma} only uses the fact that $\widetilde{\Omega} = \{\widetilde{\varphi}(x) > 0\}$ as well as the fact that $\widetilde{\varphi}$ is defined in $\mathbb{R}^{2}$. Indeed, $\varphi$ is defined in $\mathbb{R}^{2}$ after we applied Whitney-type extension theorem before, and $\varphi$ and $\widetilde{\varphi}$ are related by (\ref{tildevarphi}). So we can transfer Lemma \ref{geometriclemma} to our setting without any effort.\\

	Now, let us prove Lemma \ref{secondgradient}.

	\begin{figure}[h]
		\includegraphics[scale=0.5]{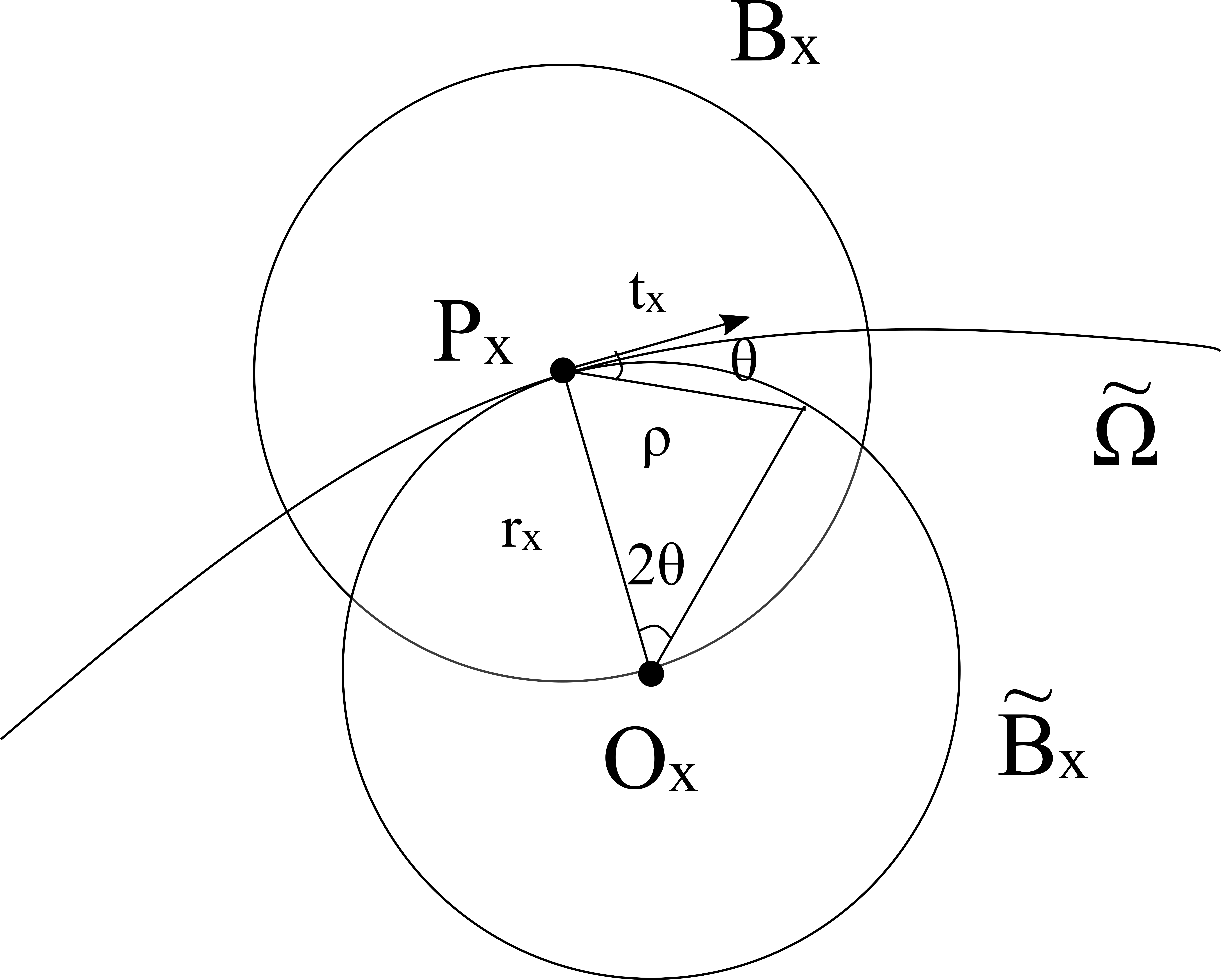}
		\caption{}
		\label{figure4}
	\end{figure}

	\begin{proof}[Proof of Lemma \ref{secondgradient}]
		Let $n_{x}, S_{x}$ be from Lemma \ref{geometriclemma} and let $\widetilde{B}_{x} := B_{r_{x}}(O_{x})$, where $O_{x} := P_{x} + r_{x}n_{x}$, $B_{x} = B_{r_{x}}(P_{x})$, see Figure \ref{figure3}. Then $P_{x} \in \partial\widetilde{B}_{x}$ and the unit inner normal vector to $\partial\widetilde{B}_{x}$ at $P_{x}$ is $n_{x}$. First note that we have $\partial\widetilde{B}_{x}\cap B_{x} \subseteq S_{x}$. Indeed, let $t_{x}$ be the vector tangent to $\partial\widetilde{B}_{x}$ at $P_{x}$, let $\theta \in [0,\frac{\pi}{2}]$ denote the angle between $t_{x}$ and $\nu$, see Figure \ref{figure4}. If $P_{x} + \rho\nu \in \partial\widetilde{B}_{x}\cap B_{x}$ then we have $\theta \leq \frac{\pi}{6}$. By law of sines, we also have
		\[
		\frac{\rho}{\sin 2\theta} = \frac{r_{x}}{\cos\theta}.
		\]
		Note that $\nu\cdot n_{x} = \sin\theta$, so $\big(\frac{\rho}{r_{x}}\big)^{\gamma} \geq 2|\nu\cdot n_{x}|$ follows immediately from the fact that $(2\sin\theta)^{\gamma} \geq 2\sin\theta$, which is true when $2\sin\theta \leq 1$.\\
		
		Combining this with Lemma \ref{geometriclemma} gives us
		\[
		(\widetilde{\Omega}\Delta\widetilde{B}_{x})\cap B_{x} \subseteq S_{x}.
		\]
		Let 
		\[
		u_{\widetilde{B}_{x}}(z) := \frac{1}{2\pi} \int_{\widetilde{B}_{x}}\frac{(z-y)^{\perp}}{|z-y|^{2}}\frac{R^{4}}{|y|^{4}}dy.
		\]
		We claim that (with a universal constant $C < \infty$),
		\begin{equation}\label{magicsymmetric}
		|\nabla^{2}u_{\widetilde{B}_{x}}(x)| \leq \frac{C}{r_{x}}.
		\end{equation}
		The proof will follow directly from Lemma \ref{symmetricmagic} below.

		By the definition of $\widetilde{v}$ and $u_{\widetilde{B}_{x}}$, we have
		\begin{equation}\label{differencedouble}
		\begin{split}
		|\nabla^{2}\widetilde{v}(x) - \nabla^{2}u_{\widetilde{B}_{x}}(x)|
		&\leq \int_{\mathbb{R}^{2}\setminus B_{x}}\frac{C}{|x-y|^{3}}dy\\
		& \quad + \int_{(\widetilde{\Omega}\Delta\widetilde{B}_{x})\cap B_{x}}\frac{C}{|x-y|^{3}}dy.
	    \end{split}
	    \end{equation}
	    By assumption of the Lemma, the distance between $x$ and $P_{x}$ (denoted by $d(x)$) is less than $\frac{1}{4}r_{x}$. Therefore the first term in the right hand side can be bounded by $Cr_{x}^{-1}$. Now we denote the second term by $I$.\\
		
		First note that 
		\[
		\text{dist}(x,S_{x}) \geq \frac{d(x)}{2}.
		\]
		Indeed, if $P_{x} + \rho\nu \in B_{d(x)/2}(x)$, then we have $-1 \leq \nu\cdot n_{x} \leq -\frac{\sqrt{3}}{2}$ and $\rho < r_{x}$, hence $P_{x} + \rho\nu \not\in S_{x}$. Also note that, when $|P_{x} - y| \geq 2d(x)$, we have
		\[
		|P_{x} - y| \leq |x - y| + d(x) \leq 2|x-y|.
		\]
		Now we deduce, with a universal constant $C < \infty$,
		
		\begin{tabbing}
			$\qquad$ $\quad$ $I$ \= $\leq$ $\displaystyle\int_{S_{x}}\frac{C}{|x-y|^{3}}dy$\\
			\> $\leq$ $\displaystyle\int_{S_{x}\setminus B_{2d(x)}(P_{x})}\frac{C}{|x-y|^{3}}dy + C(\frac{d(x)}{2})^{-3}|S_{x}\cap B_{2d(x)}(P_{x})|$\\
			\> $\leq$ $\displaystyle\int_{S_{x}\setminus B_{2d(x)}(P_{x})}\frac{C}{|P_{x}-y|^{3}}dy + \frac{C}{d(x)^{3}}|S_{x}\cap B_{2d(x)}(P_{x})|$\\
			\> $\leq$ $\displaystyle\int_{2d(x)}^{r_{x}}\frac{C}{\rho^{3}}(\frac{\rho}{r_{x}})^{\gamma}\rho d\rho + \frac{C}{d(x)^{3}}\displaystyle\int_{0}^{2d(x)}(\frac{\rho}{r_{x}})^{\gamma}\rho d\rho$\\
			\> $\leq$ $Cd(x)^{-1+\gamma}r_{x}^{-\gamma}.$
		\end{tabbing}
		
		Combining with (\ref{magicsymmetric}) and (\ref{differencedouble}), we obtain
		\[
		|\nabla^{2}\widetilde{v}(x)| \leq Cd(x)^{-1+\gamma}r_{x}^{-\gamma} + Cr_{x}^{-1}.
		\]
		The proof is complete because $d(x) \leq \frac{1}{4}r_{x}$.
		
	\end{proof}
	
	Next, we are going to prove (\ref{magicsymmetric}). To make the argument simpler and also more general, we let $B_{r} = \{x \in \mathbb{R}^{2}:|x-(0,-r)| < r\}$, $f$ be a function such that $f$, $\nabla f$ and $\nabla^{2}f$ are bounded by a universal constant $C$ in $B_{r}$, and 
	\[
	 u(x) = \dfrac{1}{2\pi}\displaystyle\int_{B_{r}}\frac{(x-y)^{\perp}}{|x-y|^{2}}f(y)dy.
	\] 
	By setting $B_{r}$ to be $\widetilde{B}_{x}$ and $f(y)$ to be $\dfrac{1}{|y|^{4}}\chi_{\widetilde{B}_{x}}$, (\ref{magicsymmetric}) can be derived from the following Lemma.
	
	\begin{lemma}\label{symmetricmagic}
		With the above notations, we have $|\nabla^{2}u(x)| \leq \dfrac{C}{r}$, for $r < 1$. where $x = (0,h)$, $h$ is any positive real number.
	\end{lemma}
	
	\begin{proof}[Proof of Lemma \ref{symmetricmagic}]
		
		First, by assumption, for any $y$ in $B_{r}$, we have
		\[
		|f(y) - f(0) - \nabla f(0)\cdot y| \leq C|y|^{2}.
		\]
		By definition, we have
		\[
		u(x) = \dfrac{1}{2\pi}\int_{B_{r}}\frac{(x-y)^{\perp}}{|x-y|^{2}}f(y)dy = I_{1} + I_{2} + I_{3},
		\]
		where 
		\begin{tabbing}
			$\qquad$ \= $I_{1} := \dfrac{1}{2\pi}\displaystyle\int_{B_{r}}\dfrac{(x-y)^{\perp}}{|x-y|^{2}}f(0)dy,$\\
			\> $I_{2} := \dfrac{1}{2\pi}\displaystyle\int_{B_{r}}\dfrac{(x-y)^{\perp}}{|x-y|^{2}}\nabla f(0)\cdot \vec{y}dy,$\\
			\> $I_{3} := \dfrac{1}{2\pi}\displaystyle\int_{B_{r}}\dfrac{(x-y)^{\perp}}{|x-y|^{2}}(f(y) - f(0) - \nabla f(0)\cdot \vec{y})dy.$
		\end{tabbing}
		
		So we just need to prove $|\nabla^{2}I_{i}(x)| \leq \dfrac{C}{r}$, for $i = 1,2,3$.\\
		
		For $i = 1$, we use the argument from \cite{KRYZ}. Observe that
		\[
		I_{1}(x) = f(0)(\nabla^{\perp}\Delta^{-1}\chi_{B_{r}})(x).
		\]
		Since $|x - (0,-r)| > r$, we have by the rotational invariance of $I_{1}(x)$ (and with $n$ the outer unit normal vector to $\partial B_{|x-(0,-r)|}\big((0,-r)\big)$)
		
		\begin{equation}\label{estimateI1}
		\begin{split}
			I_{1}(x) & = \dfrac{\big(x - (0,-r)\big)^{\perp}}{|x-(0,-r)|}|I_{1}(x)|\\
			& = \dfrac{\big(x-(0,-r)\big)^{\perp}}{|x-(0,-r)|}\displaystyle\fint_{\partial B_{|x-(0,-r)|}\big((0,-r)\big)}n\cdot f(0)\nabla\Delta^{-1}\chi_{B_{r}}d\sigma\\
 			& = \dfrac{\big(x-(0,-r)\big)^{\perp}}{|x-(0,-r)|^{2}}\dfrac{1}{2\pi}\displaystyle\int_{B_{|x-(0,-r)|}\big((0,-r)\big)}f(0)\chi_{B_{r}}(y)dy\\
			& = \dfrac{1}{2}f(0)r^{2}\dfrac{(x - (0,-r))^{\perp}}{|x-(0,-r)|^{2}}.
		\end{split}
		\end{equation}
		
		Differentiate this, then we have
		\[
		|\nabla^{2}I_{1}(x)| \leq \dfrac{C}{r}.
		\]
		
		For $i = 3$, since $|x-y| \geq |y|$, we have
		\[
		|\nabla^{2}I_{3}(x)| \lesssim \int_{B_{r}}\frac{|y|^{2}}{|x-y|^{3}}dy \leq \int_{B_{r}}\frac{1}{|y|}dy \lesssim r.
		\]
		
		For $i = 2$, first note that it suffices to control $|\nabla^{2}(k\cdot I_{2}(x))|$, for any constant vector $k = (k_{1}, k_{2})$. Denoting $a = \nabla f(0)$, we have
		
		\begin{tabbing}
			$\quad$ $2k\cdot I_{2}(x)$ \= $=$ $-2\displaystyle\int_{B_{r}}\dfrac{(x-y)\cdot k^{\perp}}{|x-y|^{2}}(a\cdot y)dy$\\
			\> $=$ $\displaystyle\int_{B_{r}}\nabla_{y}\cdot\big((\ln|x-y|^{2})k^{\perp}\big)(a\cdot y)dy$\\
			\> $=$ $\displaystyle\int_{\partial B_{r}}(n\cdot k^{\perp})\ln|x-y|^{2}(a\cdot y)dS(y)$ $-$ $(k^{\perp}\cdot a)\displaystyle\int_{B_{r}}\ln|x-y|^{2}dy$\\
			\> $:=$ $I_{21}(x) - I_{22}(x).$
		\end{tabbing}
		Here $n$ is the outer normal vector of $\partial B_{r}$. \\
		
		Controlling $\nabla^{2}I_{22}(x)$ is straightforward. Indeed, $\nabla I_{22}(x)$ is the velocity field generated by the vorticity patch $2(k^{\perp}\cdot a)\chi_{B_{r}}(x)$. By estimate (\ref{estimateI1}), we have 
		\[
		\nabla I_{22}(x) = 2(k^{\perp}\cdot a)\pi r^{2}\dfrac{(x - (0,-r))^{\perp}}{|x-(0,-r)|^{2}}.
		\]
		So $|\nabla ^{2}I_{22}(x)| \leq C$, where $C$ is a universal constant.\\
		
		For $I_{21}(x)$, note that $\partial_{22}I_{21}(x) = -\partial_{11}I_{21}(x)$ and $\partial_{12}I_{21}(x) = \partial_{21}I_{21}(x)$, so it suffices to consider two cases.
		
		\begin{equation}\label{partial11}
			\partial_{11}I_{21}(x) = 2\displaystyle\int_{\partial B_{r}}(n\cdot k^{\perp}) \dfrac{(h-y_{2})^{2} - y_{1}^{2}}{|x-y|^{4}}(a_{1}y_{1} + a_{2}y_{2})dS(y),
		\end{equation}
		\begin{equation}\label{partial12}
		\partial_{12}I_{21}(x) = 4\int_{\partial B_{r}}(n\cdot k^{\perp}) \dfrac{y_{1} (h-y_{2})}{|x-y|^{4}}(a_{1}y_{1} + a_{2}y_{2})dS(y).
		\end{equation}

		First note that when $|y| \leq r/4$, we have $y_{2} \leq \dfrac{C}{r} y_{1}^{2}$ and that $|n\cdot k - k_{2}| \leq \dfrac{C}{r}|y_{1}|$.\\
		
		For (\ref{partial11}), we have, with a constant $C < \infty$, that depends on $a$ and $k$:
		\begin{align*}
		\partial_{11}I_{21}(x) & =  2\displaystyle\int_{\partial B_{r}\cap \{|y| \geq r/4}\}(n\cdot k^{\perp}) \dfrac{(h-y_{2})^{2} - y_{1}^{2}}{|x-y|^{4}}(a_{1}y_{1} + a_{2}y_{2})dS(y)\\
		& \quad  + 2\displaystyle\int_{\partial B_{r}\cap\{|y| < r/4\}}(n\cdot k^{\perp}) \dfrac{(h-y_{2})^{2} - y_{1}^{2}}{|x-y|^{4}}a_{1}y_{1}dS(y)\\
		 & \quad + 2\displaystyle\int_{\partial B_{r}\cap\{|y| < r/4\}}(n\cdot k^{\perp}) \dfrac{(h-y_{2})^{2} - y_{1}^{2}}{|x-y|^{4}}a_{2}y_{2}dS(y)\\
		 & \leq C + 2\displaystyle\int_{\partial B_{r}\cap\{|y| < r/4\}}(n\cdot k^{\perp} - k_{2}) \dfrac{(h-y_{2})^{2} - y_{1}^{2}}{|x-y|^{4}}a_{1}y_{1}dS(y)\\
		 & \quad + 2\displaystyle\int_{\partial B_{r}\cap\{|y| < r/4\}}k_{2} \dfrac{(h-y_{2})^{2} - y_{1}^{2}}{|x-y|^{4}}a_{1}y_{1}dS(y)\\
		 & \quad + \frac{C}{r}\displaystyle\int_{\partial B_{r}\cap\{|y| < r/4\}} \dfrac{((h-y_{2})^{2} - y_{1}^{2})y_{1}^{2}}{((h-y_{2})^{2} + y_{1}^{2})^{2}}dS(y)\\
		 & \leq C + \frac{C}{r}\displaystyle\int_{\partial B_{r}\cap\{|y| < r/4\}}\dfrac{((h-y_{2})^{2} - y_{1}^{2})y_{1}^{2}}{((h-y_{2})^{2} +  y_{1}^{2})^{2}}dS(y) + 0\\ 
		 &  \quad + \frac{C}{r}\displaystyle\int_{\partial B_{r}\cap\{|y| < r/4\}} \dfrac{((h-y_{2})^{2} - y_{1}^{2})y_{1}^{2}}{((h-y_{2})^{2} + y_{1}^{2})^{2}}dS(y)\\
		 & \leq C + \frac{C}{r}\displaystyle\int_{\partial B_{r}\cap\{|y| < r/4\}}dS(y) \leq C.
		\end{align*}
	Note that $ \displaystyle\int_{\partial B_{r}\cap\{|y| < r/4\}}k_{2} \dfrac{(h-y_{2})^{2} - y_{1}^{2}}{|x-y|^{4}}a_{1}y_{1}dS(y) = 0$, because the integrand is odd in $y_{1}$.\\
		
		Similarly, we can derive that $\partial_{12}I_{21}(x) \leq C$. We leave details to interested readers.

	\end{proof}

	\section{General case}
	
    In this section, we consider the general case, where initial data is
	\[
	\omega_{0}(x) = \sum\limits_{k=1}^{N}\theta_{k}\chi_{\Omega_{k}(0)}(x).
	\]
	
	By Yudovich theory (see \cite{Yudovich}, \cite{Majda} or \cite{Marchioro}), there exists a unique solution in the form of 
	\begin{equation}\label{patchsolution}
	\omega(x,t) := \sum\limits_{k=1}^{N}\theta_{k}\chi_{\Omega_{k}(t)}(x)
	\end{equation}
	with $\Omega_{k}(t) = \Phi_{t}(\Omega_{k}(0))$ for each $k$. Also note that $\Phi_{t}(x)$ is uniquely defined for any $x \in \mathbb{R}^{2}$, due to a time independent log-Lipschitz bound
	\begin{equation}\label{loglip}
	|u(x,t) - u(y,t)| \leq C\norm{\omega_{0}}_{L^{\infty}}|x-y|\log(1 + |x-y|^{-1}).
	\end{equation}

	By definition \ref{generalpatchdef}, to show that $\omega$ in (\ref{patchsolution}) is a $C^{1,\gamma}$ patch solution, we need to prove that $\{\partial\Omega_{k}(t)\}_{k=1}^{N}$ is a family of disjoint simple closed curves for each $t \geq 0$, and 
	\[
	\sup\limits_{t \in [0,T]}\max\limits_{k}\norm{\Omega_{k}(t)}_{C^{1,\gamma}} \, < \infty
	\]
	for each $T < \infty$.\\
	
	First note that (\ref{loglip}) yields
	\[
	\min\limits_{i\not = k}\text{dist}(\Omega_{i}(t), \Omega_{k}(t)) \geq \delta(t) > 0
	\]
	for all $t \geq 0$, where $\delta(t)$ decreases at most double exponentially in time. This will ensure that the effects of the patches on each other will be controlled. Now, it remains to prove that each $\partial\Omega_{k}(t)$ is a simple closed curve with $\parallel\partial\Omega_{k}(t)\parallel_{C^{1,\gamma}}$ uniformly bounded on bounded time interval.\\
	
	Let us decompose 
	\[
	u = \sum\limits_{i=1}^{N}u_{i}
	\]
	with each $u_{i}$ coming from the contribution of the patch $\Omega_{i}$ to $u$. If $i \not = k$, then we have
	\[
	\norm{\nabla^{n}u_{i}(\cdot,t)}_{L^{\infty}(\Omega_{k}(t))} \, \leq C(\omega_{0},n)\delta(t)^{-n-1}
	\]
	for all $n \geq 0$. This yields
	\[
	\norm{\nabla u_{i}(\cdot,t)}_{\dot{C}^{\gamma}(\Omega_{k}(t))} \, \leq C(\omega_{0})\delta(t)^{-3}.
	\]
	
	Analogously to Proposition \ref{gradient o fvelocity}, we also have the estimate by simple scaling,
	\begin{equation}\label{generalgradientv}
	\norm{\nabla u_{i}(\cdot,t)}_{L^{\infty}(\mathbb{R}^{2})} \leq C_{\gamma}|\theta_{i}|\Big(1+\log_{+}\frac{A_{\gamma}(t)}{A_{\inf}(t)}\Big).
	\end{equation}
	
	With all these in hand, let us prove Theorem \ref{generalpatch}.
	\begin{proof}[Proof of Theorem \ref{generalpatch}]
	We now consider $\varphi_{k}$ and $w_{k} := \nabla^{\perp}\varphi_{k}$ for each $\Omega_{k}$. We also add $\sup\limits_{k}$ in the definitions of $A_{\infty}$, $A_{\gamma}$ and add $\inf\limits_{k}$ in the definition of $A_{\inf}$. With $\Theta := \max\limits_{1 \leq k \leq N}|\theta_{k}|$, for each $k$ and $t > 0$, we have
	
	\begin{tabbing}
		$\qquad$ $\norm{(\nabla u)w_{k}}_{\dot{C}^{\gamma}(\Omega_{k})}$ \= $\leq$ $C_{\gamma}\Theta A_{\gamma}\Big(1 + \log_{+}\dfrac{A_{\gamma}}{A_{\inf}}\Big) + A_{\infty}$\\
		\> $\quad$ $+$ $\sum\limits_{i \not = k}\norm{\nabla u_{i}}_{L^{\infty}(\Omega_{k})}\norm{w_{k}}_{\dot{C}^{\gamma}(\Omega_{k})}$\\
		\> $\quad$ $+$ $\sum\limits_{i \not = k}\norm{\nabla u_{i}}_{\dot{C}^{\gamma}(\Omega_{k})}\norm{w_{k}}_{L^{\infty}(\Omega_{k})}$\\
		\> $\leq$ $C_{\gamma}N\Theta A_{\gamma}\Big(1 + \log_{+}\dfrac{A_{\gamma}}{A_{\inf}}\Big) + C(\omega_{0})N\delta(t)^{-3}A_{\infty}.$
		
	\end{tabbing}

	Then we have estimates,
	\[
	A_{\gamma}^{\prime}(t) \leq C_{\gamma}N\Theta A_{\gamma}(t)\Big(1 + \log_{+}\frac{A_{\gamma}(t)}{A_{\inf}(t)}\Big) + C(\omega_{0})N\delta(t)^{-3}A_{\infty}(t).
	\]
	Let $\widetilde{A}(t) := A_{\gamma}(t)A_{\inf}(t)^{-1} + A_{\infty}(t)$, then a simple computation yields that 
	\[
	\widetilde{A}^{\prime}(t) \leq C(\gamma,N,\omega_{0})\widetilde{A}(t)\big(\delta(t)^{-3} + \log_{+}\widetilde{A}(t)\big).
	\]
	
	Since $\delta(t)^{-3}$ increases at most double exponentially in time, it follows that $\widetilde{A}(t)$ increases at most triple exponentially. So $\norm{\partial\Omega_{k}(t)}_{C^{1,\gamma}}$ is uniformly bounded on bounded  time intervals, thus completing the proof.\\
    \end{proof}

	\section{One special case with double exponential upper bound}
	We consider a special case in $D := B_{1}(0)$ with initial data in the following form,
	\[
	\omega_{0}(x) = \omega_{1}(x,0) - \omega_{2}(x,0) = \chi_{\Omega_{1}(0)}(x) - \chi_{\Omega_{2}(0)}(x),
	\]
	where $\Omega_{1}(0)$ and $\Omega_{2}(0)$ are two single disjoint patches that are symmetric to each other with respect to the line $x_{1} = 0$. By uniqueness, we know the solution is still of the form
	\[
	\omega(x,t) = \omega_{1}(x,t) - \omega_{2}(x,t) = \chi_{\Omega_{1}(t)}(x) - \chi_{\Omega_{2}(t)}(x),
	\]
	where $\Omega_{1}(t)$ and $\Omega_{2}(t)$ are two symmetric single disjoint patches.\\

	Let us prove Theorem \ref{symmetriccase}.
	
	\begin{proof}[Proof of Theorem \ref{symmetriccase}]
		
		From now on, we will drop $t$ from $\Omega_{k}(t)$, since the estimate is time independent. We adopt $\varphi_{k}$ and $w_{k}$ for $k = 1,2$ from Section 3 and we also add $\sup\limits_{k}$ in the definitions of $A_{\infty}$, $A_{\gamma}$ and add $\inf\limits_{k}$ in the definition of $A_{\inf}$ as we did in Section 3. Note that $u = u_{1} + u_{2}$, and for $i = 1, 2$,
		\[
		u_{i} = v_{i} + \widetilde{v}_{i} = -\dfrac{1}{2\pi}\int_{\Omega_{i}}\dfrac{(x-y)^{\perp}}{|x-y|^{2}}dy + \dfrac{1}{2\pi}\int_{\widetilde{\Omega}_{i}}\dfrac{(x-y)^{\perp}}{|x-y|^{2}}\dfrac{1}{|y|^{4}}dy.
		\]
		
		From the proof in Section 3, our goal is to estimate $\norm{(\nabla u)w_{k}}_{\dot{C}^{\gamma}(\Omega_{k})}$ for $k = 1, 2$. Without loss of generality, if suffices to prove the case in which $k = 1$. As we discussed above, $\norm{(\nabla u)w_{1}}_{\dot{C}^{\gamma}(\Omega_{1})}$ can be decomposed to be sum of $\norm{(\nabla u_{1})w_{1}}_{\dot{C}^{\gamma}(\Omega_{1})}$ and $\norm{(\nabla u_{2})w_{1}}_{\dot{C}^{\gamma}(\Omega_{1})}$.
		
		For the first term $\norm{(\nabla u_{1})w_{1}}_{\dot{C}^{\gamma}(\Omega_{1})}$, $u_{1}$ and $w_{1}$ are both generated by the patch $\Omega_{1}$, it is exactly same to the single patch case in Section 2. Thus we have,
		\[
		\norm{(\nabla u_{1})w_{1}}_{\dot{C}^{\gamma}(\Omega_{1})} \leq C_{\gamma}A_{\gamma}\Big(1+\log_{+}\frac{A_{\gamma}}{A_{\inf}} \Big) + C_{\gamma}A_{\infty}.
		\]
		
		For the second term $\norm{(\nabla u_{2})w_{1}}_{\dot{C}^{\gamma}(\Omega_{1})}$, we decompose $(\nabla u_{2})w_{1}$ as 
		\[
		(\nabla u_{2})w_{1} = (\nabla v_{2})w_{1} + (\nabla\widetilde{v}_{2})w_{1}.
		\]
		
		First we claim that $\norm{(\nabla v_{2})w_{1}}_{\dot{C}^{\gamma}(\Omega_{1})}$ can be bounded by $CA_{\gamma}\Big(1+\log_{+}\dfrac{A_{\gamma}}{A_{\inf}}\Big)$. Indeed, $v_{2}$ can be regarded as the velocity field generated by the patch $\Omega_{2}$ in $\mathbb{R}^{2}$. $w_{1}$ is a divergence free vector field that is tangent to the boundary of $\Omega_{1}$, which is symmetric to $\Omega_{2}$ over $x_{1} = 0$. This case has been treated in \cite{KRYZ} from page 15 to the end of the Section 3. Note that the symmetry here is with respect to $x_{1} = 0$, while in \cite{KRYZ} the symmetry is with repect to $x_{2} = 0$.

		For the second term $\norm{(\nabla \widetilde{v}_{2})w_{1}}_{\dot{C}^{\gamma}(\Omega_{1})}$, as we treated single patch case in Section 2, we split $D$ as the union of $B_{10/11}(0))$ and $A(0;9/10,1)$. In $B_{10/11}(0))$, we know the velocity away from $\partial D$ is smooth enough, so $\norm{(\nabla\widetilde{v}_{2})w_{1}}_{\dot{C}^{\gamma}(\Omega_{1}\cap B_{10/11}(0))}$ is bounded by $C_{\gamma}(A_{\gamma} + A_{\infty})$. In $A(0;9/10,1)$, with $g(x) := \nabla\widetilde{v}_{2}w_{1}(x)$, we have
		
		\begin{equation}\label{Cgammaofg}
		\begin{split}
		\dfrac{|g(x)-g(y)|}{|x-y|^{\gamma}} & \leq |\nabla\widetilde{v}_{2}(y)|\norm{w_{1}}_{\dot{C}^{\gamma}(\Omega_{1})}\\
		& \quad + \dfrac{|\nabla\widetilde{v}_{2}(x) - \nabla\widetilde{v}_{2}(y)|}{|x-y|^{\gamma}}|w_{1}(x)|.
		\end{split}
		\end{equation}

		The first term in (\ref{Cgammaofg}) can be easily bounded by $C_{\gamma}A_{\gamma}\Big(1+\log_{+}\dfrac{A_{\gamma}}{A_{\inf}}\Big)$, because of (\ref{generalgradientv}) and definition of $A_{\gamma}$ and $w_{k}$.
		
        For the second term, we need a claim similar to Proposition \ref{T2estimate}. First we note that in (\ref{Cgammaofg}) $x, y$ are in $\Omega_{1}$, and $\widetilde{v}_{2}$ is generated by the patch $\Omega_{2}$. Similar to Section 2, we need to introduce some notations here. For any $x \in \Omega_{1}$, define $d(x) := \text{dist}(x,\widetilde{\Omega}_{2})$. Let $P_{x} \in \partial\widetilde{\Omega}_{2}$ be the point such that $d(x) = \text{dist}(x,P_{x})$ (if there are multiple such points, we pick any one of them). We define $\widetilde{w}_{2}(x)$ to be $\nabla^{\perp}\widetilde{\varphi}_{2}(x)$, where $\widetilde{\varphi}_{2}(x) = \varphi_{2}(\widetilde{x})$. By using a similar argument, we have, for $x, y \in \Omega_{1}\cap A(0;9/10,1)$ and $d(x) \leq d(y)$, 
        
        \begin{equation}\label{symmetriccaseinequality}
        \dfrac{|\nabla\widetilde{v}_{2}(x) - \nabla\widetilde{v}_{2}(y)|}{|x-y|^{\gamma}} \leq C_{\gamma}\Big(1 + \log_{+}\dfrac{A_{\gamma}}{A_{\inf}}\Big)\min\Big\{\dfrac{A_{\gamma}}{|\widetilde{\omega}_{2}(P_{x})|}, d(x)^{-\gamma}\Big\}
        \end{equation}

		The symmetry of $\Omega_{1}$ and $\Omega_{2}$ implies that we can choose
		\[
		\varphi_{1}(x) = \varphi_{2}(\bar{x}),
		\]
		where $\bar{x} = (-x_{1}, x_{2})$. Therefore, by definition $w_{i} = \nabla^{\perp}\varphi_{i}$, we have
		\[
		w_{2}(x) = -\overline{w_{1}(\bar{x})}.
		\]
		For any $x \in \Omega_{1}$,
		\[
		|w_{1}(x)| \leq |w_{1}(x) - w_{1}(\overline{\widetilde{P}_{x}})| + |w_{1}(\overline{\widetilde{P}_{x}})|.
		\]
		From above, we know $|w_{1}(\overline{\widetilde{P}_{x}})| = |w_{2}(\widetilde{P}_{x})|$. Since $x \in A(0; 9/10, 1)$, the symmetry over $\partial D$ is very close to it over a line. So we have $|x - \overline{\widetilde{P}_{x}}| \leq Cd(x)$. Thus we obtain,
		\[
		|w_{1}(x)| \leq CA_{\gamma}d^{\gamma} + C|\widetilde{w}_{2}(P_{x})|.
		\]
		Therefore,
		\[
		\norm{(\nabla\widetilde{v}_{2})w_{1}}_{\dot{C}^{\gamma}(\Omega_{1})} \leq C_{\gamma}A_{\gamma}\Big(1+\log_{+}\dfrac{A_{\gamma}}{A_{\inf}}\Big) + A_{\infty}.
		\]
		Thus, we get
		\[
		\norm{(\nabla u)w}_{\dot{C}^{\gamma}(\Omega_{1})} \leq C_{\gamma}A_{\gamma}\Big(1+\log_{+}\dfrac{A_{\gamma}}{A_{\inf}}\Big) + A_{\infty}.
		\]
		We conclude that $A_{\gamma}$, $A_{\inf}$ and $A_{\infty}$ grow at most double exponentially in time.
		
	\end{proof}

	\section{Example with double exponential growth}

	In this section, we are going to use the example constructed in\cite{KS} to show that the upper bound obtained by the previous section is actually sharp.\\
	
	First we introduce some notations that will be adopted throughout this section. With $\phi$ to be the usual angular variable, we have
	\begin{align*}
	D &:= B_{1}(e_{2}), \quad \textrm{with} \quad e_{2} = (0,1),\\
	D^{+} & := \{ (x_{1},x_{2}) \in D : x_{1} \geq 0\},\\
    D_{1}^{\gamma} & :=  \{(x_{1}, x_{x}) \in D^{+}|\frac{\pi}{2}-\gamma \geq \phi \geq 0\},\\
	D_{2}^{\gamma} & :=  \{(x_{1}, x_{2}) \in D^{+}|\frac{\pi}{2} \geq \phi \geq \gamma\},\\
	Q(x_{1},x_{2}) & :=  \{(y_{1}, y_{2})\in D^{+}|y_{1} \geq x_{1}, y_{2} \geq x_{2}\},\\
	\Omega(x_{1},x_{2},t) & :=  \dfrac{4}{\pi}\displaystyle\int_{Q(x_{1},x_{2})}\frac{y_{1}y_{2}}{|y|^{4}}\omega(y,t)dy.
	\end{align*}

	Consider two-dimensional Euler equation on $D$, let $\omega$ be vorticity. We will take smooth patch initial data $\omega_{0}$ ao that $\omega_{0}(x) \geq 0$ for $x_{1} > 0$ and $\omega_{0}$ is odd in $x_{1}$. It can be easily checked that such symmetry will be conserved in time $t$. Let us state the Key Lemma (see \cite{KS} for Lemma 3.1).
	
	\begin{lemma}\label{keylemma}
		Take any $\gamma$, $\pi/2 > \gamma > 0$. Then there exists $\delta > 0$ such that
		\[
		u_{1}(x) = -x_{1}\Omega(x_{1},x_{2}) + x_{1}B_{1}(x), \hspace{1mm}|B_{1}| \leq C(\gamma)\norm{\omega_{0}}_{L^{\infty}}, \hspace{2mm} \forall x \in D_{1}^{\gamma}, |x| \leq \delta
		\]
		\[
		u_{2}(x) = x_{2}\Omega(x_{1},x_{2}) + x_{2}B_{2}(x), \hspace{1mm}|B_{2}| \leq C(\gamma)\norm{\omega_{0}}_{L^{\infty}}, \hspace{2mm} \forall x \in D_{2}^{\gamma}, |x| \leq \delta
		\]
	\end{lemma}

	Note that, in \cite{KS} Lemma \ref{keylemma} applies only to smooth $\omega$, but at the same time the argument can extend to patches without any effort. Following the proof in \cite{KS}, exponential growth of curvature can be achieved easily. Indeed, take initial data $\omega_{0}(x)$ which is equal to $1$ everywhere in $D^{+}$ except on a thin strip of width equal to $\delta/2$ ($\delta$ is chosen from Lemma \ref{keylemma}, for some small $\gamma < \pi/10$) near the vertical axis $x_{1} = 0$, where $\omega_{0}(x) = 0$. Then we round the corner of this single patch to make the boundary smooth and also guarantee that $\omega_{0} = 1$ everywhere in $D^{+}$ except on a thin strip of width equal to $\delta$ near the vertical axis $x_{1} = 0$. Denote the patch in $D^{+}$ at the initial time by $P(0)$, so $\omega_{0}(x) = \chi_{P(0)}(x) - \chi_{\bar{P}(x)}(x)$, where $\bar{P} := \big\{(x_{1}, x_{2}): (-x_{1},x_{2}) \in P\big\} $. By odd symmetry, two single patches $P$ and $\bar{P}$ will stay in the two half disks respectively for all time $t$. Due to incompressibility, the measure of the set where $\omega(x,t) = 0$ does not exceed $4\delta$. In this case, for every $x \in D^{+}$ with $|x| < \delta$, we can derive the following estimate for $\Omega(x_{1}, x_{2})$,
	\[
	\Omega(x_{1},x_{2},t) \geq \int_{2\delta}^{2}\int_{\pi/6}^{\pi/3}\omega(r,\phi)\frac{\sin 2\phi}{2r}d\phi dr \geq \frac{\sqrt{3}}{4}\int_{2\delta}^{2}\int_{\pi/6}^{\pi/3}\frac{\omega(r,\phi)}{r}.
	\]
	The value of the integral on the right hand side is minimal when the area where $\omega(r,\phi) = 0$ is seated around small values of the radial variable. Since this area does not exceed $4\delta$, we have
	\begin{equation}\label{logdelta}
	  \frac{4}{\pi}\Omega(x_{1},x_{2},t) \geq c_{1}\int_{c_{2}\sqrt{\delta}}^{1}\int_{\pi/6}^{\pi/3}\frac{1}{r}d\phi dr \geq C_{1}\log \delta^{-1},
	\end{equation}
	where $c_{1}$, $c_{2}$ and $C_{1}$ are positive universal constants.\\
	
	By Lemma \ref{keylemma}, we have that for all $|x| \leq \delta$, $x \in D^{+}$ that lie on the patch boundary, with universal constants $C_{1}, C_{2} < \infty$)
	\[
	 u_{1}(x, t) \leq -x_{1}(C_{1}\log\delta^{-1} - C_{2}).
	\]
	We can choose $\delta > 0$ sufficiently small so that $u_{1}(x,t) \leq -x_{1}$ for all time if $|x| < \delta$. Due to the boundary condition on $u$, the trajectories that start at the boundary will stay on the boundary for all time. Taking the trajectory starting at the leftmost point $x_{0}$, such that $x_{0} \in \partial D\cap\partial P(0)$, with $|x| < \delta$, we get 
	\begin{equation}\label{exponentialgrowth}
	\Phi_{t}^{1}(x_{0}) \leq x_{0}^{1}e^{-t},
	\end{equation}
	where $\Phi_{t}(x)$ is defined in (\ref{particletrajectory}) and superscript denotes the corresponding coordinate of the vector.	Note that, for a curve, one can use the distance for it to change direction by $\pi/2$ to characterize the curvature. In this case, $\Phi_{t}^{1}(x_{0})$ is being pushed towards the origin from the right along the boundary $\partial D$, and by odd symmetry, the axis $x_{1} = 0$ is a barrier that patch $P$ can not pass. So $\partial P$ has to form a round corner at $\Phi_{t}^{1}(x_{0})$ when approaching the origin. Exponential growth of curvature therefore follows from (\ref{exponentialgrowth}).\\

	To achieve double exponential growth on curvature, by discussion above, we need an example where we can track a point on the patch boundary that approaches the origin at double exponential speed.\\
	
	Let us prove Theorem \ref{sharpexample} by using the example constructed in \cite{KS}, the proof is similar to \cite{KS}.

	\begin{figure}[h]
		\includegraphics[scale=0.5]{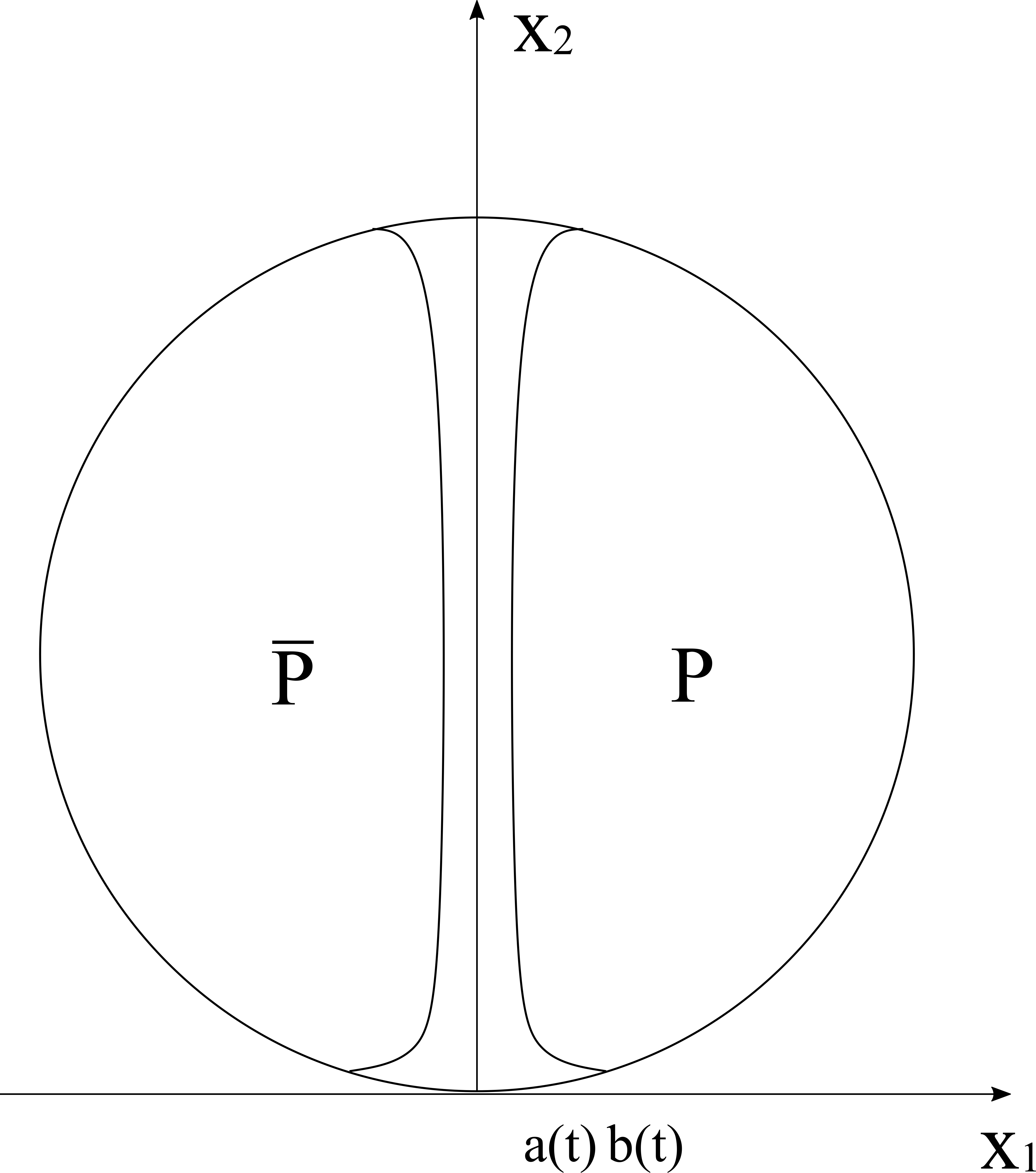}
		\caption{}
		\label{figure5}
    \end{figure}

	\begin{proof}[Proof of Theorem \ref{sharpexample}]
	We first fix some small $\gamma > 0$ (take $\pi/10$ for example). We will first take the smooth patch initial data $\omega_{0}$ as the one constructed in the previous exponential growth example, with $\omega_{0} = 1$ everywhere in $D^{+}$ except on a thin strip of width equal to $\delta$ near the vertical axis $x_{1} = 0$ (we will make some modifications later to achieve double exponential growth). We choose $\delta > 0$ small enough such that Lemma \ref{keylemma} applies and that $C_{1}\log\delta^{-1} > 100C(\gamma)$ with $C_{1}$ from (\ref{logdelta}) and $C(\gamma)$ from Lemma \ref{keylemma}.\\
	
	For $0 < x^{\prime}_{1}, x^{\prime\prime}_{1} < 1$, we denote
	\[
	 \mathcal{O}(x^{\prime}_{1},x^{\prime\prime}_{1}) = \big\{(x_{1}, x_{2})\in D^{+}|x^{\prime}_{1} \leq x \leq x^{\prime\prime}_{1}, x_{2} < x_{1}\big\}.
	\]
	For $0 < x_{1} < 1$, we let
	\begin{equation}\label{fastslowspeed}
    \begin{split}
	\overline{u}(x_{1},t) &= \max\limits_{(x_{1},x_{2}) \in D^{+}x_{2} < x_{1}}u_{1}(x_{1},x_{2},t), \\
	\underline{u}(x_{1},t) &= \min\limits_{(x_{1},x_{2}) \in D^{+}x_{2} < x_{1}}u_{1}(x_{1},x_{2},t),
	\end{split}
	\end{equation}

	and define $a(t)$, $b(t)$ by
	\begin{align*}
	a^{\prime}(t) &= \overline{u}(a(t),t),  \quad a(0) = \epsilon^{10},\\
	b^{\prime}(t) &= \underline{u}(b(t),t), \quad b(0) = \epsilon.	
	\end{align*}
	
	Like in \cite{KS}, by choosing $\epsilon < \delta/2$ small enough such that $-\log\epsilon$ is larger than some universal constant that appears in \cite{KS}, then
	\[
	  a(t) \leq \epsilon^{8\exp(t/2\pi)}.
	\]

	 From the discussion above, we can finally pick our initial data $\omega_{0}(x)$ to be $1$ everywhere in $D^{+}$ except on a thin strip of width equal to $\epsilon^{10}/2$ near the vertical axis $x_{1} = 0$, and round the corner of the patch to make its boundary smooth. The double exponential growth of curvature follows as in the argument for exponential growth above. See Figure \ref{figure5}.

	\end{proof}

    \section*{Acknowledgement}
	I would like to thank Alexander Kiselev and Yao Yao for their valuable advices. I also would like to acknowledge  partial support of the NSF-DMS grant 1412023.

\end{document}